\documentclass[12pt]{amsart}
\usepackage{amscd,amsmath,amsthm,amssymb}
\usepackage[left]{lineno}
\usepackage{color}
\usepackage{stmaryrd}
\SetSymbolFont{stmry}{bold}{U}{stmry}{m}{n}

\usepackage[utf8]{inputenc}
\usepackage{cleveref}

\usepackage{graphicx}
\usepackage{epstopdf} 

\usepackage{graphicx}
\usepackage{xcolor}
\usepackage{subeqnarray}

\usepackage{cases}
\usepackage{subcaption}
\usepackage{cite}
\usepackage{tikz}

\definecolor{verylight}{gray}{0.97}
\definecolor{light}{gray}{0.9}
\definecolor{medium}{gray}{0.85}
\definecolor{dark}{gray}{0.6}

\def\NZQ{\mathbb}               

\def\ZZ{{\NZQ Z}}

\def\m{\mathfrak{m}}

 \def\ab{{\mathbf a}}
 \def\xb{{\mathbf x}}

\def\G{{\mathcal G}}

\def\ab{{\mathbf a}}

\def\xb{{\mathbf x}}

\def\0b{{\mathbf 0}}
\def\alphab{{\mathbf \alphab}}
\def\c_ib{{\mathbf c_i}}

\def\reg{{\mathbf reg}}
\def\height{\operatorname{ht}}
\def\depth{\operatorname{depth}}
\def\opn#1#2{\def#1{\operatorname{#2}}} 
\opn\chara{char} \opn\length{\ell} \opn\pd{pd} \opn\rk{rk}
\opn\projdim{proj\,dim} \opn\injdim{inj\,dim} \opn\rank{rank}
\opn\depth{depth} \opn\grade{grade} \opn\height{height}
\opn\embdim{emb\,dim} \opn\codim{codim}

\opn\Tr{Tr} \opn\bigrank{big\,rank}
\opn\superheight{superheight}\opn\lcm{lcm}
\opn\trdeg{tr\,deg}
\opn\reg{reg} \opn\lreg{lreg} \opn\ini{in} \opn\lpd{lpd}
\opn\size{size} \opn\sdepth{sdepth}
\opn\link{link}\opn\fdepth{fdepth}\opn\lex{lex}
\opn\tr{tr}
\opn\type{type}
\opn\gap{gap}
\opn\arithdeg{arith-deg}
\opn\HS{HS}
\opn\GL{GL}
\opn\div{div} \opn\Div{Div} \opn\cl{cl} \opn\Cl{Cl}
\opn\Spec{Spec} \opn\Supp{Supp} \opn\supp{supp} \opn\Sing{Sing}
\opn\Ass{Ass} \opn\Min{Min}\opn\Mon{Mon}
\opn\Ann{Ann} \opn\Rad{Rad} \opn\Soc{Soc}\opn\Deg{Deg}
\opn\Im{Im} \opn\Ker{Ker} \opn\Coker{Coker} \opn\Am{Am}
\opn\Hom{Hom} \opn\Tor{Tor} \opn\Ext{Ext} \opn\End{End}
\opn\Aut{Aut} \opn\id{id}

\opn\nat{nat}
\opn\pff{pf}
\opn\Pf{Pf} \opn\GL{GL} \opn\SL{SL} \opn\mod{mod} \opn\ord{ord}
\opn\Gin{Gin} \opn\Hilb{Hilb}\opn\sort{sort}
\opn\PF{PF}\opn\Ap{Ap}
\opn\mult{mult}
\opn\bight{bight}
\opn\aff{aff}
\opn\relint{relint} \opn\st{st}
\opn\lk{lk} \opn\cn{cn} \opn\core{core} \opn\vol{vol}  \opn\inp{inp} \opn\nilpot{nilpot}
\opn\link{link} \opn\star{star}\opn\lex{lex}\opn\set{set}
\opn\width{wd}
\opn\Fr{F}
\opn\QF{QF}
\opn\G{G}
\opn\type{type}\opn\res{res}
\opn\conv{conv}
\opn\Ind{Ind}
\opn\gr{gr}

\def\pot#1#2{#1[\kern-0.28ex[#2]\kern-0.28ex]}

\opn\dirlim{\underrightarrow{\lim}}
\opn\inivlim{\underleftarrow{\lim}}
\let\to=\rightarrow

\def\Implies{\ifmmode\Longrightarrow \else
	\unskip${}\Longrightarrow{}$\ignorespaces\fi}
\def\implies{\ifmmode\Rightarrow \else
	\unskip${}\Rightarrow{}$\ignorespaces\fi}
\def\iff{\ifmmode\Longleftrightarrow \else
	\unskip${}\Longleftrightarrow{}$\ignorespaces\fi}

\let\:=\colon
\newtheorem{Theorem}{Theorem}[section]
\newtheorem{Lemma}[Theorem]{Lemma}
\newtheorem{Corollary}[Theorem]{Corollary}
\newtheorem{Proposition}[Theorem]{Proposition}
\newtheorem{Remark}[Theorem]{Remark}

\newtheorem{Definition}[Theorem]{Definition}

\newtheorem{Conjecture}[Theorem]{Conjecture}

\let\epsilon\varepsilon
\let\kappa=\varkappa
\def\qed{\ifhmode\textqed\fi
	\ifmmode\ifinner\quad\qedsymbol\else\dispqed\fi\fi}
\def\textqed{\unskip\nobreak\penalty50
	\hskip2em\hbox{}\nobreak\hfil\qedsymbol
	\parfillskip=0pt \finalhyphendemerits=0}
\def\dispqed{\rlap{\qquad\qedsymbol}}

\opn\dis{dis}
\def\pnt{{\raise0.5mm\hbox{\large\bf.}}}

\opn\Lex{Lex}

\newcommand*{\circled}[1]{\lower.7ex\hbox{\tikz\draw (0pt, 0pt)%
		circle (.5em) node {\makebox[1em][c]{\small #1}};}}

\begin{document}

\title{Homological shifts of powers of a complementary edge ideal}	
	\author{Dancheng Lu, Zexin Wang, Guangjun zhu}

\address{Dancheng Lu, School of Mathematical Sciences,\allowbreak
Soochow University,\allowbreak
215006 Suzhou,\allowbreak
P.R.China}
\email{ludancheng@suda.edu.cn}

\address{Zexin Wang, School of Mathematical Sciences,\allowbreak
Soochow University,\allowbreak
215006 Suzhou,\allowbreak
P.R.China}
\email{zexinwang6@outlook.com}

\address{Guangjun Zhu, School of Mathematical Sciences,\allowbreak
Soochow University,\allowbreak
215006 Suzhou,\allowbreak
P.R.China}
\email{zhuguangjun@suda.edu.cn}

		


		
		

		\begin{abstract}

The homological shift algebra and the projective dimension function of a complementary edge ideal are investigated. Let $G$ be a connected graph, and let $I$ be its complementary edge ideal. For bipartite graphs $G$, we show that the projective dimension of $I^s$ increases strictly with $s$ until reaching its maximum value. For trees and cycles, explicit expressions for the projective dimension of $I^s$ are provided, along with precise formulas of their homological shift ideals. In particular, it is shown that the $i$-th homological shift algebra of such ideals is generated in degree at most $i$. We prove that if $G$ is a tree, then the homological shift ideal $\mathrm{HS}_i(I^i)$, when divided by a suitable monomial, is a Veronese-type ideal. Moreover, every Veronese-type ideal can be realized in this manner.

		\end{abstract}
		\setcounter{tocdepth}{1}
		
		\maketitle

\section{Introduction}

Let \(S:=\mathbb{K}[x_1,\ldots,x_n]\) denote a polynomial ring over a field \(\mathbb{K}\). In the study of syzygies and free resolutions of monomial ideals in $S$, the concept of \emph{homological shift ideals} serves as a powerful algebraic tool. While this concept first appeared implicitly in the 2005 monograph of Miller and Sturmfels \cite{MS}, it was not formally defined until Herzog et al. in 2021 \cite{HMRZ021a}.

Given a vector \( \ab = (a_1, \ldots, a_n) \in \ZZ_{\ge 0}^{n} \), let \( \xb^{\ab}\) denote the monomial \(x_1^{a_1} \cdots x_n^{a_n} \) in
\( S\). For a monomial ideal \( I \subset S \), we adopt the definition introduced in \cite{HMRZ021a} and define the \( i \)-th homological shift ideal of \( I \) (for each integer \( i \geq 0 \)) as
\[
\HS_i(I):= \left( \xb^{\ab} \mid \Tor_{i}^{S}(\mathbb{K}, I)_{\ab} \neq 0 \right).
\]
A central question in the study of homological shift ideals concerns the identification of which properties of a monomial ideal \(I\) are preserved by all its homological shift ideals \(\mathrm{HS}_i(I)\). A key property to investigate in this regard is polymatroidality: if \(I\) is a polymatroidal ideal, do all \(\mathrm{HS}_i(I)\) inherit this property? Partial results addressing this question have appeared in the literature: Bayati \cite{Ba1} proved that all \(\mathrm{HS}_i(I)\) of a matroidal ideal \(I\) remain matroidal; Herzog et al. \cite{HMRZ021a} showed that the homological shift ideals of Veronese-type ideals are polymatroidal; Ficarra \cite{F1} established that \(\mathrm{HS}_1(I)\) is polymatroidal whenever \(I\) is polymatroidal. Recently, a complete and positive answer to this question has been established in \cite{CMS}. Another  property of interest pertains to linear quotients: if \(I\) admits linear quotients, do all \(\mathrm{HS}_i(I)\) inherit this property? This holds for \(i=1\), as demonstrated in \cite{FH2023}; however, it fails to hold in general for \(i\ge2\). Multiple  counterexamples have been constructed in the literature to confirm this fact, and \cite[Proposition 3.5]{LW} provides one such instance.

Of equal, if not greater, significance to the above problem is the explicit construction of homological shift ideals for specific families of monomial ideals. For instance, if \(I\) is an equigenerated square-free Borel ideal, then \(\mathrm{HS}_1(I)\) has been explicitly characterized in \cite{HMRZ023}; if \(I\) is the edge ideal of a co-chordal graph or a Veronese-type ideal, the structure of \(\mathrm{HS}_i(I)\) has been rigorously determined in \cite{HMRZ021a}; if \(I\) is the generalized co-letterplace ideal of a poset, \(\mathrm{HS}_i(I)\) has been explicitly described in \cite{LW}. Additional advances in the study of homological shift ideals can be found in \cite{Ba2, F3, RS, TBR}.

To study the asymptotic properties of homological shift ideals, Ficarra and Qureshi introduced the $i$-th \emph{homological shift algebra} \( \mathrm{HS}_i(\mathcal{R}(I)) \) in \cite{FQ2}. By definition, this abelian group structure is given, for all \( i\geq 1 \), by:
\[
\mathrm{HS}_i(\mathcal{R}(I)) = \bigoplus_{s \geq 0} \mathrm{HS}_i(I^s).
\]
When \( I \) has linear powers (meaning \( I^s \) admits a linear resolution for all \( s \geq 1 \)), Ficarra and Qureshi showed in \cite[Proposition 1.2]{FQ2} that \( I \cdot \mathrm{HS}_i(I^s) \subseteq \mathrm{HS}_i(I^{s+1}) \) for all \( s \geq 0 \). This containment implies \( \mathrm{HS}_i(\mathcal{R}(I)) \) is a graded \( \mathcal{R}(I) \)-module.
Furthermore, by  \cite[Theorem 1.4]{FQ2}, this graded module is finitely generated. The generating degree of homological shift algebras over \( \mathcal{R}(I) \) has been investigated in \cite{FQ1} and \cite{FL}. In \cite{FQ1}, it was shown that if \( I \) is the edge ideal of a simple graph, then an explicit expression for \( \mathrm{HS}_1(I^s) \) can be derived, from which it follows that \( \mathrm{HS}_1(\mathcal{R}(I)) \) is generated in degree \( 1 \).
In \cite{FL}, Ficarra and Lu proposed the following conjecture: if \( I \) is a polymatroidal ideal, then \( \mathrm{HS}_i(\mathcal{R}(I)) \) is generated in degrees at most \( i \) as an \( \mathcal{R}(I) \)-module for all \( i > 0 \) (see \cite[Conjecture 5.1]{FL}). They verified this conjecture for special cases, including \( i = 1 \) and specific classes of polymatroidal ideals such as principal Borel ideals, polymatroidal ideals satisfying the strong exchange property, and matroidal edge ideals. Another class of polymatroidal ideals satisfying the conjecture was provided in \cite{CHL}.
It is natural to reformulate \cite[Conjecture 5.1]{FL} in a more general sense:
\begin{Conjecture}\label{conjecture} If \( I \) is a monomial ideal for which $I^s$ admits linear quotients for all $s\geq 1$, then \( \mathrm{HS}_i(\mathcal{R}(I)) \) is generated in degrees at most \( i \) over \( \mathcal{R}(I) \) for all \( i\geq 1 \).
\end{Conjecture}

Let \(I \subset S\) be a non-zero square-free monomial ideal generated in degree \(d\). Note that \(1\le d\le n\). If \(d=1\) or \(d=n\), \(I\) has a simple structure, leaving no non-trivial problems for further investigation.

Within the range \(2\le d\le n-2\), two distinguished values of \(d\) (\(d=2\) and \(d=n-2\)) exhibit distinctive behavior relative to other degrees, particularly regarding the linear resolutions and linear quotient property of powers of such ideals.

For \(d=2\), Fröberg’s classic result \cite{Froberg88} asserts that \(I\) has a linear resolution if and only if \(I\) is the edge ideal of a co-chordal graph. Herzog, Hibi, and Zheng further extended this equivalence to  powers of such ideals \cite[Theorem 3.2]{HHZ2004}: for square-free monomial ideals of degree 2, \(I\) has a linear resolution if and only if \(I^s\) has linear quotients for all \(s\ge 1\).

A monomial ideal \(I\subset S\) is \emph{fully supported} if every variable \(x_k\) (\(1\le k\le n\)) divides some minimal monomial generator of \(I\). For the fully supported square-free monomial ideals of degree \(d=n-2\), Ficarra and Moradi \cite[Corollary 3.2]{FM1} and Hibi et al. \cite[Theorem 2.2]{HQM} independently established a parallel characterization: \(I\) has a linear resolution if and only if \(I^s\) has linear quotients for all \(s\ge 1\).

By contrast, for intermediate degrees \(3\le d\le n-3\), \cite[Theorem D]{FM1} shows that there exists a fully supported square-free monomial ideal \(I\subset S\) of degree \(d\) such that \(I\) has linear quotients but \(I^2\) fails to admit a linear resolution. Moreover, \cite[Theorem C]{FM1} shows that there exists a fully supported square-free monomial ideal \(I\subset S\) generated in degree \(d\) that fails to admit a linear resolution when \(\operatorname{char}(\mathbb{K})=2\), yet it admits one otherwise.

For \(d=2\), every square-free monomial ideal of degree \(d\) is an edge ideal, a class with a long history of extensive investigation. In sharp contrast, the class of fully supported square-free monomial ideals of degree \(d=n-2\) has only recently garnered considerable attention in combinatorial commutative algebra, owing to independent work by two teams in 2025: Ficarra and Moradi \cite{FM1, FM2, FM3} and Hibi et al. \cite{HQM}. Notably, every such ideal is the complementary edge ideal of a finite simple graph \(G\) on \(n\) vertices, defined as follows. Let \(G\) be a simple graph with vertex set \([n]=\{1,2,\dots,n\}\) and edge set \(E(G)\). Set \(\xb_{[n]}:=\prod_{i=1}^n x_i\); the complementary edge ideal of \(G\) is then
\[
I_c(G):=\left(\frac{\xb_{[n]}}{x_ix_j}: \{i,j\}\in E(G)\right).
\]
Ficarra and Moradi \cite[Corollary 3.2]{FM1} and Hibi et al. \cite[Theorem 2.2]{HQM} further proved that \(I_c(G)\) has a linear resolution if and only if \(G\) is connected.

Beyond characterizing linear resolution conditions for these ideals and their powers, the two teams independently established the sequentially Cohen-Macaulay, Cohen-Macaulay, and Gorenstein properties of \(I_c(G)\). Ficarra and Moradi extended these findings by characterizing the nearly Gorenstein and matroidal subclasses of complementary edge ideals. Additionally, they computed key homological invariants of the powers of these ideals, including regularity and depth, among others.

In this paper, we investigate the algebraic and homological properties of the complementary edge ideal of a simple graph and its powers. Specifically, we focus on determining the homological shift algebra and the projective dimension function of the complementary edge ideals of trees and cycles,  with particular emphasis on their links to Conjecture~\ref{conjecture}.

Our first result is not directly related to the other results in this paper, but is of independent interest.

\begin{Theorem}{\em (Theorem~\ref{lone})}\label{02}
Let \( I \subseteq S = \mathbb{K}[x_1, \ldots, x_n] \) be a monomial ideal with a \( d \)-linear resolution. Then for all \( i \in [n] \), we have
\[
\mathrm{HS}_i(I) + \mathrm{HS}_{i-1}(\mathfrak{m}I) = \mathfrak{m}^{[i]}I.
\]
\end{Theorem}

Next, we investigate  the projective dimension of powers of a complementary edge ideal.  Let $\mathrm{pd} (-)$ denote the projective dimension of an ideal.

 \begin{Theorem}{\em (Theorem~\ref{Cycle} and Theorem~\ref{TreeMain})}\label{0.1}
Let $G$ be a simple graph and let $I_c(G)$ be its complementary edge ideal.
\begin{itemize}
		\item[(1)] If $G$ is a tree  with $n$ vertices, then
		 \[
			\mathrm{pd}(I_c(G)^s)=\left\{
   \begin{array}{ll}
  s, & \hbox{ \text{if\ } $s\in  [n-2]$;} \\
  n-2, & \hbox{ \text{if\ } $s\geq n-2$.}
  \end{array}
  \right.
		\]
		\item[(2)] If $G$ is an even cycle with $2m$ vertices, then
		\[
\mathrm{pd}(I_c(G)^s)=\left\{
 \begin{array}{ll}
  2s, & \hbox{\text{ if\ } $s\in  [m-1]$;} \\
  2m-2, & \hbox{\text{ if\ } $s\geq m-1$.}
   \end{array}
   \right.
\]
		\item[(3)] If $G$ is  an odd cycle with $2m+1$ vertices, then
\[
\mathrm{pd}(I_c(G)^s)=\left\{
 \begin{array}{ll}
 2s, & \hbox{\text{ if\ } $s\in  [m]$;} \\
  2m, & \hbox{\text{ if\ } $s> m$.}
  \end{array}
   \right.
\]
\end{itemize}
\end{Theorem}

\begin{Proposition}\label{04} {\em (Proposition~\ref{strict})}
Let \( G \) be a connected graph on vertex set \( [n] \). Then \( \mathrm{pd}(I_c(G)^s) \) is strictly increasing in \( s \) before reaching \( n-2 \).
\end{Proposition}

If \( G \) is bipartite, \( \mathrm{pd}(I_c(G)^s) \) remains constant at \( n-2 \) for all \( s \) beyond the point where it first reaches \( n-2 \).  If \( G \) is non-bipartite, the maximal value of \( \mathrm{pd}(I_c(G)^s) \) over \( s\geq1 \) is \( n-1 \). It remains undetermined whether \( \mathrm{pd}(I_c(G)^s) \) jumps  to \( n-1 \) right  after attaining \( n-2 \).

Given an ideal $I$, we let $\mathrm{dstab}(I)$ denote the smallest integer $k$ such that $\mathrm{pd}(I^s)$ remains  constant for all integers $s\geq k$.

\begin{Theorem}\label{05} {\em (Corollary~\ref{dstab} and Theorem~\ref{uni2})} Let $G$ be a connected  graph with $n$ vertices. Then
\[
\mathrm{dstab} (I_c(G))\leq n-2.
\]
\end{Theorem}

The main theme of this paper is to  describe the  homological shift algebra of  a complementary edge ideal.
\begin{Theorem}{\em (Theorem~\ref{HS1generated})} \label{first}
Let $G$ be a connected graph on vertex set  $[n]$ and  let $I=I_c(G)$ be its complementary edge ideal. Then the 1-st homological shift algebra $\HS_1(\mathcal{R}(I))$ is generated in degree $1$ as a graded $\mathcal{R}(I)$-module.
\end{Theorem}

Let $\alpha$ denote the square-free monomial $x_1x_2\cdots x_n$.
\begin{Theorem}\label{07} {\em (Theorem~\ref{3.1})} Let $G$ be a tree with $n$ vertices, labelled as in Remark \ref{labeling}, and let $I=I_c(G)$ be its complementary edge ideal. Define a map $\phi: [n-1]\rightarrow [n]$ such that $\phi(j)$ is the unique $k\in [n]$, with $k>j$ for which $x_jx_k\in E(G)$. Then, for any $1\leq i\leq s$, we have
\[
\mathrm{HS}_i(I^s)=I^{s-i}\cdot \left(\frac{\alpha^i}{\prod_{k\in F}x_{\phi(k)}}\middle|\; F\subseteq [n-2] \text{\ and\ } |F|=i \right).
\]
In particular, the \( i \)-th homological shift algebra \( \mathrm{HS}_i(\mathcal{R}(I)) \) is generated exactly in degree \( i \) for all \( 1 \leq i \leq n-2 \).
\end{Theorem}

Let \(d\) be a positive integer, and \(\boldsymbol{c}=(c_1,\dots,c_n)\in\mathbb{Z}_{\geq0}^n\) be a given non-negative integer vector.
The \emph{Veronese-type ideal} of \(S\) associated with \(d\) and \(\boldsymbol{c}\) is a monomial ideal generated by all monomials in \(S\) of total degree \(d\) whose exponents are componentwise bounded by \(\boldsymbol{c}\). In symbolic form:
\[
I_{\boldsymbol{c},d}
=\Bigl(
x_1^{a_1}x_2^{a_2}\cdots x_n^{a_n}
\ \Bigm|\
\sum_{i=1}^n a_i=d,\quad 0\le a_i\le c_i\ \text{for all}\ 1\le i\le n
\Bigr).
\]

Veronese-type ideals have been extensively studied, as exemplified by \cite{FL,HHV,HRV,HMRZ021a}. We identify the following notable connection between this class of ideals and the homological shift ideals of powers of a complementary edge ideal.

\begin{Theorem}{\em (Proposition~\ref{ith homology} and Theorem~\ref{AllVeronese})}\label{08} Let $G$ be a tree, and let $I=I_c(G)$ be its complementary edge ideal. Then $\HS_i(I^i)$, up to  division by  a suitable monomial, is a Veronese-type ideal.  Furthermore, every Veronese-type ideal can be realized in this manner.
\end{Theorem}

Studying the homological shift algebra of the complementary edge ideal of a cycle is much more difficult. Regarding this topic, our results are as follows:

\begin{Theorem}
{\em (Theorem~\ref{express1})}
\label{09}
Let $G$ be a cycle of length $n$ and $I$ be its complementary edge ideal. Then, for all integers $1 \leq i <n$ and $s \geq 1$, the following equality holds:
\begin{equation*}
\mathrm{HS}_i(I^s)=\sum_{\max\left\{i - \left\lceil \frac{n}{2} \right\rceil + 1, 0\right\} \leq k \leq \left\lfloor \frac{i}{2} \right\rfloor}
\left(
  \sum_{F \subseteq [n],\ |F| = i - 2k}
  \frac{\alpha^{i - k}}{{\bf x}_F} \cdot I^{s - i + k}
\right).
\end{equation*}
Here, we use the convention that $I^t=0$ if $t<0$, and ${\bf x}_F$ denotes the monomial $\prod_{i\in F}x_i$.
\end{Theorem}

By Theorem~\ref{09}, if $s<\left\lfloor\frac{i}{2}\right\rfloor$, then $s-i+k<0$ for all $k\le\left\lfloor\frac{i}{2}\right\rfloor$, whence $\mathrm{HS}_i(I^s)=0$. This conclusion is consistent with parts (2) and (3) of Theorem~\ref{0.1}.

\begin{Theorem} {\em (Corollary~\ref{degree})} \label{10} Let $G$ be a cycle of length $n$ and  let $I_c(G)$ be its complementary edge ideal, then  $\mathrm{HS}_i(\mathcal{R}(I_c(G)))$ is generated in degree at most $i-q$ for all $i> 0$. Here, $q:=\max\left\{i - \left\lceil \frac{n}{2} \right\rceil + 1, 0\right\}$.
\end{Theorem}

Theorems~\ref{07} and \ref{10}  show that Conjecture~\ref{conjecture} holds for the complementary edge ideal of a tree or a cycle, while Theorem~0.7 also shows that the conclusion of this conjecture cannot be improved; that is, $\mathrm{HS}_i(\mathcal{R}(I))$ cannot be generated in a degree smaller than $i$ in general.

The article is organized as follows: Section \ref{sec:prelim} provides the definitions and basic facts that will be used throughout the paper. In Section~\ref{independent}, we prove Theorem~\ref{02} and discuss several of its consequences. In Section \ref{sec:project}, we compute the projective dimension of powers of the complementary edge ideals of trees and cycles, and prove Proposition~\ref{04} and Theorem~\ref{05}.  In Section \ref{sec:homological}, we are devoted to obtaining explicit expressions for the homological shift ideals of powers of the complementary edge ideals of connected graphs, whereby Theorems \ref{first}–\ref{10} are established.

\section{Preliminaries and Basic Definitions}\label{sec:prelim}
This section collects the definitions and basic facts used throughout the paper. We first retain the notation introduced in the Introduction, then introduce additional notation, core concepts, and key lemmas for our study.

\subsection{Basic Notation}
For a vector $\mathbf{a} = (a_1, \ldots, a_n) \in \mathbb{Z}^n$, we set $\mathbf{x}^{\mathbf{a}} := x_1^{a_1} \cdots x_n^{a_n}$. For a subset $F \subseteq [n]$, we set $\mathbf{x}_F := \prod_{j \in F} x_j$ if $F \neq \emptyset$, and $\mathbf{x}_F := 1$ if $F = \emptyset$. For positive integers $j \leq k$, let $[j, k]$ denote the set $\{j, j+1, \ldots, k\}$; in particular, $[1, k]$ is abbreviated as $[k]$. We use the symbol $\sqcup$ to denote the disjoint union.

\subsection{Core Concepts: Homological Shift Ideals and Algebras}
Given the central role of both the homological shift ideal and algebra in this paper, we restate their definitions, drawn respectively from \cite{HMRZ021a} and \cite{FQ1}.

\begin{Definition}\rm
Let $I\subset S$ be a monomial ideal with minimal multigraded free $S$-resolution
\[
\mathbb{F}\ \ :\ \ 0 \rightarrow F_p \rightarrow F_{p-1} \rightarrow \cdots\rightarrow F_0 \rightarrow I\rightarrow 0,
\]
where $F_i = \bigoplus\limits_{j=1}^{\beta_i(I)}S(-\mathbf{a}_{i,j})$. The vectors $\mathbf{a}_{i,j}\in \mathbb{Z}^n$ are called the $i$-th \emph{multigraded shifts} of the resolution $\mathbb{F}$. The monomial ideal
    $$
    \mathrm{HS}_i(I):=\ (\mathbf{x}^{\mathbf{a}_{i,j}}:\ j=1,\dots,\beta_i(I))
    $$
is called the $i$-th \textit{homological shift ideal} of $I$.
\end{Definition}

Note that $\mathrm{HS}_0(I)=I$ and $\mathrm{HS}_i(I)=(0)$ for $i<0$ or $i>\pd(I)$. If a monomial ideal admits a {\it $d$-linear resolution}, then $\HS_i(I)$ is either $0$ or generated in degree $i+d$ for all $i$.

\begin{Definition}\rm
Let $I\subset S$ be a monomial ideal. We define its $i$-th \emph{homological shift algebra} $ \mathrm{HS}_i(\mathcal{R}(I)) $ as
\[
\mathrm{HS}_i(\mathcal{R}(I)):= \bigoplus_{s \geq 0} \mathrm{HS}_i(I^s).
\]
\end{Definition}

In general, $ \mathrm{HS}_i(\mathcal{R}(I)) $ is merely an abelian group. However, if $ I^s $ admits a linear resolution for all $ s\geq 1 $, then $ \mathrm{HS}_i(\mathcal{R}(I)) $ is a finitely generated graded $ \mathcal{R}(I) $-module.

\subsection{Linear Quotients and Related Properties}
Linear quotients are a key tool for characterizing the structure of homological shift ideals. We recall its definition and a fundamental formula below.

Let $ I $ be a monomial ideal, and let $ \mathcal{G}(I) $ denote its unique minimal set of monomial generators. We say that $ I $ has \emph{linear quotients} if there exists an ordering $ u_1, \ldots, u_m $ of $ \mathcal{G}(I) $ such that the colon ideal $ (u_1, \ldots, u_{i-1}) : (u_i) $ is generated by variables for all $ i = 2, \ldots, m $. By \cite[Lemma 2.1]{JZ}, we may assume without loss of generality that $ \deg(u_1) \leq \cdots \leq \deg(u_m) $.

When $ I $ admits linear quotients, we define $ \set(u_1) = \emptyset $ and
\[
\set(u_i) = \left\{ \ell \ :\ x_\ell \in (u_1, \ldots, u_{i-1}) : (u_i) \right\}, \quad \text{for } i = 2, \ldots, m.
\]

By \cite[Lemma 1.5]{ET}, it follows that
\begin{equation}\label{eq:HerTak1}
\HS_i(I) = \left( \mathbf{x}_F u :\ u \in \mathcal{G}(I),\ F \subseteq \set(u),\ |F| = i \right),
\end{equation}
where $ \mathbf{x}_F = \prod_{j \in F} x_j $ if $ F \neq \emptyset $, and $ \mathbf{x}_\emptyset = 1 $ otherwise. From this formula, we deduce that the projective dimension of $ I $ satisfies
\begin{equation}\label{eq:HerTak2}
\mathrm{pd}(I) = \max\left\{ |\set(u)| :\ u \in \mathcal{G}(I) \right\}.
\end{equation}

\subsection{Complementary edge ideal and FM condition}
Let $G$ be a finite simple graph with vertex set $V(G)=\{1,\ldots , n\}$ and edge set $E(G)$. The \emph{complementary edge ideal} $I_c(G)$ of $G$ is defined as the monomial ideal
\[
I_c(G)=\left(\frac{\mathbf{x}_{[n]}}{x_ix_j}:\ \{i,j\}\in E(G)\right).
\]

A key characterization of \(I_c(G)\) is as follows: \(I_c(G)\) admits a linear resolution if and only if \(G\) is connected, if and only if \(I_c(G)^s\) admits linear quotients for all \(s\geq 1\). In \cite{FM1}, the ordering of minimal generators of  \(I_c(G)^s\) that yields linear quotients was given as follows.

For a simple graph \(G=(V,E)\) and a vertex subset \(U\subseteq V\), the \(\emph{induced subgraph}\) of \(G\) on \(U\), denoted \(G[U]\), is the graph with vertex set \(U\) and edge set \(\{\{u,v\}\in E\mid u,v\in U\}\).

\begin{Definition}\label{FQC}\em
Let \(G\) be a connected graph with \(n\) vertices. A labeling of \(V(G)\) with \([n]\) is said to satisfy the \emph{FM condition} if, for all \(i\in [n-1]\), the induced subgraph \(G[V(G)\setminus [i]] = G[\{i+1,\dots,n\}]\) is connected. This condition is introduced in the work of Ficarra and Moradi \cite[Section 3]{FM1}; we adopt the name \emph{FM condition} here for brevity, following the initials of the authors.
\end{Definition}

As noted in \cite[Remark 3.3]{FM1}, if \(V(G)\) admits a labeling by \([n]\) that satisfies the \(\boldsymbol{\text{FM}}\) condition, then \(I_c(G)^s\) admits linear quotients for all \(s\geq 1\) with respect to the lexicographical order on the minimal generators of \(I_c(G)^s\) induced by \(x_1 > x_2 > \cdots > x_n\). Throughout this paper, any ordering of the minimal monomial generators of an ideal \(I\) with respect to which \(I\) admits linear quotients refers to this specific lexicographical order.

\subsection{Key Lemmas}
The following lemma can be proved in the same way as \cite[Lemma 10.3.5]{HHBook}.
\begin{Lemma}\label{HH}
Let $I, J \subset S$ be monomial ideals that admit $d$-linear resolutions and suppose that $J \subset I$.
Then, for all $\mathbf{a} = (a_1,\ldots,a_n) \in \mathbb{N}^n$ and all $i \ge 0$, we have
\[
\beta_{i,\mathbf{a}}(J) \le \beta_{i,\mathbf{a}}(I).
\]
In particular, for all $i \ge 0$, we have $\HS_i(J) \subseteq \HS_i(I)$.
\end{Lemma}

Lemma~\ref{HH} yields the following two immediate corollaries, which are crucial to our subsequent arguments.

\begin{Lemma}\label{HSsubset}
Let $G$ be a connected graph and $H$ be a connected subgraph of $G$. Let $u:=\prod_{i\in V(G)\setminus V(H)}x_i$. Then $u^s\cdot\HS_i(I_c(H)^s) \subseteq \HS_i(I_c(G)^s)$ for all $i,s\geq 1$.
In particular, $\mathrm{pd}(I_c(H)^s)\leq \mathrm{pd}(I_c(G)^s)$.
\end{Lemma}
\begin{proof}
Note that $u^sI_c(H)^s\subseteq I_c(G)^s$. By the characterization of complementary edge ideals, both $u^sI_c(H)^s$ and $I_c(G)^s$ admit linear resolutions. The result thus follows immediately from Lemma~\ref{HH}.
\end{proof}

\begin{Lemma}\label{pdnondecreasing}
Let $G$ be a connected graph. Then $\beta_{i}(I_c(G)^s)\leq \beta_{i}(I_c(G)^{s+1})$ for all $s\geq 1$. In particular, the projective dimension function $\mathrm{pd}(I_c(G)^s)$ of $I_c(G)$ is a non-decreasing function of $s$.
\end{Lemma}
\begin{proof}
Fix a minimal monomial generator $u$ of $I_c(G)$. Then $uI_c(G)^s\subseteq I_c(G)^{s+1}$, and both ideals admit \((n-2)(s+1)\)-linear resolutions. It follows that $\beta_{i}(I_c(G)^s)=\beta_{i}(uI_c(G)^s)\leq \beta_{i}(I_c(G)^{s+1})$ by Lemma~\ref{HH}.
\end{proof}
	
\section{A Theorem on the homological shift ideals of monomial ideals with linear resolutions }
    \label{independent}

    In this section, we present a result regarding the homological shift ideal of a monomial ideal with a linear resolution.  Although this result is not directly related to other parts of this paper, it is interesting and applies to general monomial ideals. We therefore include it here.

   Let $\m$ be the maximal graded ideal of $S=\mathbb{K}[x_1,\ldots,x_n]$, and let
$\m^{[i]}$ denote the $i$-th square-free power of $\m$, i.e., $\m^{[i]}=({\bf x}_F:\ F\subseteq [n], |F|=i).$

    \begin{Theorem}\label{lone}
Let \( I \subseteq S = \mathbb{K}[x_1, \ldots, x_n] \) be a monomial ideal with a \( d \)-linear resolution. Then for all \( i \in [n] \), we have
\[
\mathrm{HS}_i(I) + \mathrm{HS}_{i-1}(\mathfrak{m}I) = \mathfrak{m}^{[i]}I.
\]
\end{Theorem}

\begin{proof}
Let \( T := \bigoplus\limits_{{\bf x}^\alpha \in \mathcal{G}(I)} \mathbb{K}[-\alpha] \), then we obtain  the short exact sequence of multigraded \( S \)-modules:
\[
0 \rightarrow \mathfrak{m}I \rightarrow I \rightarrow T \rightarrow 0.
\]

Fix an integer \( i \in [n] \) and a multidegree \( \beta \in \mathbb{N}^n \) with \( |\beta| = i + d \). Since \( I \) has a \( d \)-linear resolution, its \( k \)-th syzygy module is concentrated in degree \( d + k \) for all \( k \geq 0 \), so \( \mathrm{Tor}_{i-1}(I, \mathbb{K})_{\beta} = 0 \). For \( \mathfrak{m}I \), note that multiplying an ideal with a \( d \)-linear resolution by \( \mathfrak{m} \) yields an ideal with a \( (d+1) \)-linear resolution, so \( \mathrm{Tor}_{i}(\mathfrak{m}I, \mathbb{K})_{\beta} = 0 \).

Applying the \(\mathbb{K}\otimes_S-\) functor to the short exact sequence above, we obtain the following long exact sequence of multigraded \( \mathbb{K} \)-vector spaces (truncated at the relevant degrees):
\begin{align*}
&\rightarrow \mathrm{Tor}_i(\mathfrak{m}I, \mathbb{K})_{\beta}
  \rightarrow \mathrm{Tor}_i(I, \mathbb{K})_{\beta}
  \rightarrow \mathrm{Tor}_i(T, \mathbb{K})_{\beta} \rightarrow \\
&\quad \rightarrow \mathrm{Tor}_{i-1}(\mathfrak{m}I, \mathbb{K})_{\beta}
  \rightarrow \mathrm{Tor}_{i-1}(I, \mathbb{K})_{\beta} \rightarrow.
\end{align*}
Since \( \mathrm{Tor}_{i-1}(I, \mathbb{K})_{\beta} = \mathrm{Tor}_{i}(\mathfrak{m}I, \mathbb{K})_{\beta} = 0 \), this reduces to a short exact sequence:
\[
0 \rightarrow \mathrm{Tor}_i(I, \mathbb{K})_{\beta} \rightarrow \mathrm{Tor}_i(T, \mathbb{K})_{\beta} \rightarrow \mathrm{Tor}_{i-1}(\mathfrak{m}I, \mathbb{K})_{\beta} \rightarrow 0.
\]
Thus, $\mathrm{Tor}_i(T,\mathbb{K})_{\beta}\neq 0$ if and only if either $\mathrm{Tor}_i(I,\mathbb{K})_{\beta}\neq 0$ or $\mathrm{Tor}_{i-1}(\m I,\mathbb{K})_{\beta}\neq 0$. Since all minimal generators of \( \mathrm{HS}_i(I) \), \( \mathrm{HS}_{i-1}(\mathfrak{m}I) \), and \( \mathrm{HS}_i(T) \) have degree \( i + d \) (matching \( |\beta| \)), we conclude that
\[
\mathrm{HS}_i(I) + \mathrm{HS}_{i-1}(\mathfrak{m}I) = \mathrm{HS}_i(T).
\]

On the other hand, the Koszul complex of \( x_1, \ldots, x_n \) provides a minimal free resolution of \( \mathbb{K} \). It follows that for each \( {\bf x}^\alpha \in \mathcal{G}(I) \), \( \mathrm{Tor}_i(\mathbb{K}[-\alpha], \mathbb{K})_{\beta} \neq 0 \) if and only if \( \beta = \alpha + \gamma \) for some \( \gamma \in \mathbb{N}^n \) with \( {\bf x}^\gamma \in \mathcal{G}(\mathfrak{m}^{[i]}) \).  Since \( T = \bigoplus_{{\bf x}^\alpha \in \mathcal{G}(I)} \mathbb{K}[-\alpha] \),  \( \mathrm{HS}_i(T) \) is generated by all monomials \( {\bf x}^{\alpha + \gamma} = {\bf x}^\alpha {\bf x}^\gamma \) where \( {\bf x}^\alpha \in \mathcal{G}(I) \) and \( {\bf x}^\gamma \in \mathcal{G}(\mathfrak{m}^{[i]}) \), which is exactly \( \mathfrak{m}^{[i]}I \).

Combining the two equalities, we get \( \mathrm{HS}_i(I) + \mathrm{HS}_{i-1}(\mathfrak{m}I) = \mathfrak{m}^{[i]}I \).
\end{proof}

    This result indicates  that every minimal generator of $\mathrm{HS}_i(I)$ has the form ${\bf x}_F \cdot u$, where $u\in \mathcal{G}(I)$ and $F$ is a subset of  cardinality $i$ in  $[n]$.  This  is consistent with the case when $I$ has linear quotients. In \cite{FH2023}, the authors expected  that  if $I$ has linear resolutions then $\mathrm{HS}_1(I)$ would have linear quotients (see \cite[the third paragraph, p134]{FH2023}).
This result lends support to this expectation.

   Another consequence of Theorem~\ref{lone} is as follows:
    \begin{Corollary} Let $I\subset S$  be a monomial ideal with a $d$-linear resolution. Then
     $\mathrm{HS}_{n-1}(\m I)=\m^{[n]}I$.
    In particular, if $I\neq 0$, then $\mathrm{pd}(\m I)=n-1$.
    \end{Corollary}
    \begin{proof}  Note that  $\mathrm{HS}_{n}(I)=0$.  Specializing $i$ to $n$ in Theorem~\ref{lone}  yields the desired formula. \end{proof}

\section{Projective dimension functions}
\label{sec:project}
In this section, we investigate the projective dimension of powers of complementary edge ideals for several classes of connected graphs, including trees and cycles.

First, we will review the concept of \emph{even-connectedness}  and a related result. By abuse of notation, an edge of $G$ refers to  both a two-element subset $\{i,j\}$ of $V(G)$ and a monomial $x_ix_j$. This abuse of notation should not cause any confusion.

Recall that a \emph{walk} in a graph \(G\) is a finite sequence of vertices \(v_0, v_1, \dots, v_k\) such that for every \(i = 1, 2, \dots, k\), there is an edge between \(v_{i-1}\) and \(v_i\) in \(G\); such a walk is compactly denoted by $v_0-v_1-\cdots-v_k$. The \emph{length} of this walk is defined to be the number of edges it contains, i.e., $k$.  A walk is called a \emph{closed walk} if its starting vertex coincides with its ending vertex, i.e., $v_0=v_k$. Vertices and edges can be repeated in a walk. By contrast,  a \emph{path}  in a graph \(G\) is a walk where all vertices are distinct.

\begin{Definition} \em (\cite[Definition 6.2]{AB})  Two vertices $j$ and $k$ ($j$ may be equal to $k$) of $G$ are said to be {\it even-connected } with respect to the $s$-fold product $e_1\cdots e_s$ of edges if there exists  a walk $j_1-j_2-\cdots-j_{2\ell+2}$ ($\ell \geq 2$) in $G$ such that
\begin{enumerate}
    \item $j_1=j$ and $j_{2\ell+2}=k$;
   \item For all $ t\in [\ell]$, $j_{2t}j_{2t+1}=e_i$ for some $i$;
   \item For all $i\in [s]$,  $|\{ t\in [\ell]:\ j_{2t}j_{2t+1}=e_i\}|\leq |\{ t\in [s]:\ e_t=e_i\}|$.
\end{enumerate}
\end{Definition}

\begin{Theorem}  {\em(\cite[Theorems 6.1 and  6.7]{AB})}\label{AB}
 Let $G$ be a graph and let  $e_1\cdots e_s$ be an $s$-fold product of edges in $G$, where $s\geq 1$ and $e_i$ may be equal to  $e_j$ for $i\neq j$. Then $(I(G)^{s+1}: e_1\cdots e_s)$ is generated by monomials of degree two. Moreover, each generator $x_jx_k$ ($x_j$ may be equal to $x_k$) of $(I(G)^{s+1}: e_1\cdots e_s)$
is either an edge of $G$ or even-connected with respect to $e_1\cdots e_s$.
\end{Theorem}

The following result plays a key role in this section and next section. For convenience, we always denote $\prod_{i\in V(G)}x_i$ by $\alpha(G)$ or simply  $\alpha$.

\begin{Lemma} \label{Basic} Let $G$ be a connected graph whose vertices is labeled in a way satisfying the FM condition. Let
$\beta\in \mathcal{G}(I_c(G)^s)$, and write $\beta=\frac{\alpha^s}{u_1\cdots u_s}$, where each  $u_i\in \mathcal{G}(I(G))$. Then
\[
\mathrm{set}(\beta)\subseteq \{i\mid x_i \mbox{ divides } u_1\cdots u_s \}.
\]
 Furthermore, if $u_1=x_ix_j$, then $i\in \mathrm{set}(\beta)$ if and only if  there is some $k\in [n]$ with $k>i$ such that either $\{j,k\}\in E(G)$ or $j$ is even-connected to $k$ with respect to the $(s-1)$-fold product $u_2\cdots u_s$.
\end{Lemma}

\begin{proof} For any monomials $U_1,\ldots,U_k, U$ that divides $\alpha^s$, one has $(\frac{\alpha^s}{U_1}, \cdots, \frac{\alpha^s}{U_k}):\frac{\alpha^s}{U}=(U:U_1)+\cdots+(U:U_k)$. The first assertion follows from this observation.  Next, we prove the second assertion.

Let \( u := u_1\cdots u_s \). If \( i \in \mathrm{set}(\beta) \), then there exists some \( \gamma > \beta \) such that \( \gamma = \alpha^s / v \) and \( \gamma : \beta = u : v = (x_i) \), where \( v \in \mathcal{G}(I(G)^s) \).

Since \( u = x_i x_j u_2 \cdots u_s \), we deduce \( v = x_k x_j u_2 \cdots u_s \) for some \( k \in [n] \). As \( \gamma > \beta \), it follows that \( u > v \), which in turn implies \( i < k \). It follows directly that \( x_k x_j \in (I(G)^s : u_2 \cdots u_s) \). By Theorem~\ref{AB}, either \( \{j, k\} \in E(G) \) or \( j \) is even-connected to \( k \) with respect to the \( (s-1) \)-fold product \( u_2 \cdots u_s \).

Conversely, suppose there exists \( k \in [n] \) such that \( k > i \), and either \( \{j, k\} \in E(G) \) or \( j \) is even-connected to \( k \) with respect to the \( (s-1) \)-fold product \( u_2 \cdots u_s \). Then \( x_k x_j \in (I(G)^s : u_2 \cdots u_s) \). Let \( v := x_k x_j u_2 \cdots u_s \) and \( \gamma:= \alpha^s / v \). It follows that \( \gamma > \beta \) and \( \gamma \in \mathcal{G}(I_c(G)^s) \). Moreover, \( \gamma : \beta = (x_i) \), which implies \( i \in \mathrm{set}(\beta) \).
\end{proof}

By \cite[Propositions 4.3, 4.5]{FM3}, if $G$ is a connected graph, then the maximal value of $\mathrm{pd}(I_c(G)^s)$ for $s\geq 1$   is  $n-1$ if $G$ is non-bipartite, and $n-2$ if $G$ is bipartite.
We first determine the projective dimension function of the complementary edge ideal of a cycle.

 \begin{Theorem} \label{Cycle}
{\em (1)} Let \( G \) be an even cycle with \( 2m \) vertices. Then
\[
\mathrm{pd}(I_c(G)^s) = \left\{
\begin{array}{ll}
2s, & \hbox{if } s \in [m-1]; \\
2m - 2, & \hbox{if } s \geq m - 1.
\end{array}
\right.
\]

{\em (2)} Let \( G \) be an odd cycle with \( 2m + 1 \) vertices. Then
\[
\mathrm{pd}(I_c(G)^s) = \left\{
\begin{array}{ll}
2s, & \hbox{if } s \in [m]; \\
2m, & \hbox{if } s > m.
\end{array}
\right.
\]
\end{Theorem}

\begin{proof}
(1) Suppose the vertices of \( G \) are labeled consecutively as \( 1, 2, 3, \dots, 2m-1, 2m \) in a cyclic order. It is clear that this labeling satisfies the FM condition (see Definition~\ref{FQC}).  Fix an integer \( s \in [m-1] \). Then, for any \( \beta = \frac{\alpha^s}{u_1u_2\cdots u_s} \in \mathcal{G}(I_c(G)^s) \) where each \( u_i \in \mathcal{G}(I(G)) \), Lemma~\ref{Basic} ensures that $\mathrm{set}(\beta) \subseteq \left\{i \mid x_i \text{ divides } u_1u_2\cdots u_s\right\}$. It follows immediately that
\( |\mathrm{set}(\beta)| \leq 2s \). Consequently, by formula~(\ref{eq:HerTak2}), we have
\[
\mathrm{pd}(I_c(G)^s) = \max\left\{|\mathrm{set}(\beta)| : \beta \in \mathcal{G}(I_c(G)^s)\right\} \leq 2s.
\]

Conversely, let \( \mathbf{e}_i:= x_{2i-1}x_{2i} \) for each \( i \in [s] \), and define \( \gamma:= \frac{\alpha^s}{\mathbf{e}_1\mathbf{e}_2\cdots \mathbf{e}_s} \). Then \( \gamma \in \mathcal{G}(I_c(G)^s) \). An application of Lemma~\ref{Basic} gives \( 2i-1 \in \mathrm{set}(\gamma) \) for all \( i \in [s] \). Next, note that vertex \( 2m \) is even-connected to vertex \( 2i-1 \) with respect to \( \mathbf{e}_1\mathbf{e}_2\cdots \mathbf{e}_{2i-3} \), via the walk:
\[
2m - 1 - 2 - 3 - 4 - \cdots - (2i-3) - (2i-2) - (2i-1).
\]
By virtue of this adjacency walk and the inequality \( 2m > 2i \), a further appeal to Lemma~\ref{Basic} yields \( 2i \in \mathrm{set}(\gamma) \). We thus have \( \mathrm{set}(\gamma) = [2s] \), from which it follows that
$\mathrm{pd}(I_c(G)^s) = \max\left\{|\mathrm{set}(\beta)| : \beta \in \mathcal{G}(I_c(G)^s)\right\} \geq 2s.$

Combining the upper and lower bounds, we establish the equality $\mathrm{pd}(I_c(G)^s) = 2s$ for all integers \( s \in [m-1] \).

Since $\mathrm{pd}(I_c(G)^s)$ is a non-decreasing function of $s$ by Lemma~\ref{HSsubset}, while $\mathrm{pd}(I_c(G)^s)\leq 2m-2$ for all $s\geq 1$ by \cite[Proposition 4.3]{FM3}, it follows that  $\mathrm{pd}(I_c(G)^s)=2m-2$ for all $s\geq m-1$.

The proof of (2) is similar to that of (1), so we omit the proof.
\end{proof}

Next, we turn our attention to trees. To begin with, we assign $V(G)$ a labeling for an arbitrary tree $G$ that satisfies the FM condition  as specified in Definition~\ref{FQC}.

\begin{Remark} \label{labeling}\em
Let \( G \) be a tree with \( n \) vertices. We fix a vertex of degree one as the root and label it \( n \).
To label the remaining \( n-1 \) vertices with \( [n-1]\), we use a distance-based hierarchical rule as follows:

For any vertex \( u \in V(G) \), let \( \text{dist}(u, n) \) denote the length of the unique shortest path from \( u \) to \( n \). Group \( V(G) \setminus \{n\} \) into layers \( L_d = \{ u \mid \text{dist}(u, n) = d \} \) for \( d \geq 1 \). Denote \( m := \max\{ d \mid L_d \neq \emptyset \} \). Assign labels in \( [n-1] \) from the farthest layer to the closest, that is, from \( L_m \) down to \( L_1 \): assign the smallest consecutive labels to \( L_m \), the next set to \( L_{m-1} \), and so on. The labeling order within the same layer is arbitrary.

It is easy to verify that such a labeling satisfies the FM condition. Furthermore, a key property holds: For each \( j \in [n-1] \), there exists a unique \(j< k\leq  n \) such that \( j \) and \( k \) are adjacent. We call \( k \)  the ``parent" of \( j \) relative to root \( n \).

This labeling is generally non-unique since labels within a layer can be permuted, but is unique if and only if \( G \) is a path. In the remainder of this paper, all trees are labeled in this manner.

\end{Remark}

\begin{Lemma} \label{new} Let $G$ be a tree and let $i-y_1-y_2-\cdots-y_p-k$ be a walk in $G$ with $i<k$. If $y_1<i$, then there exists $2\leq q\leq p$ such that $y_q=i$.
\end{Lemma}
\begin{proof} First, define a function \( \phi: [n-1] \to [n] \) where \( \phi(j) \) denotes the unique parent of \( j \) (as defined in Remark \ref{labeling}): for each \( j \in [n-1] \), \( \phi(j) \) is the unique vertex \( k > j \) that is adjacent to \( j \).
A fundamental observation about \( \phi \) is that for any walk in \( G \) starting at a vertex \( j \in [n-1] \) and ending at a vertex \( \ell > j \), the parent \( \phi(j) \) must lie on this walk. This follows from the tree structure: the unique path from \( j \) to \( \ell \) (and hence any walk from \( j \) to \( \ell \), which reduces to this path) must pass through \( \phi(j) \),  the unique neighbor of \( j \) with a larger label.

Applying this observation to the walk $y_1-y_2-\cdots-y_p-k$ and noting that $\phi(y_1)=i$, the result follows.
\end{proof}

\begin{Lemma} \label{Tree}
Let $G$ be a tree and $\beta\in \mathcal{G}(I_c(G)^s)$. Write $\beta=\frac{\alpha^s}{u_1\cdots u_s}$, where each $u_i\in \mathcal{G}(I(G))$; assume that $u_1=x_ix_j$ with $1\leq i<j\leq n$. Then
\begin{enumerate}
  \item $i\in \mathrm{set}(\beta)$ if and only if $j\neq n$.
  \item If $j\in \mathrm{set}(\beta)$, then there exists some $2\leq t\leq s$ such that $u_t=x_jx_k$ and $k>j$.
  \item $\mathrm{set}(\beta)\subseteq [n-2]$. In particular, $\mathrm{pd}(I_c(G)^s)\leq n-2$ for all $s\geq 1$.
\end{enumerate}
\end{Lemma}

\begin{proof}
(1) If $j\neq n$, then there exists $k\in [n]$ such that $k>j$ and $k$ is adjacent to $j$. Since $j>i$, we have $k>i$; by Lemma~\ref{Basic}, $i\in \mathrm{set}(\beta)$.

If $j=n$, then $i=n-1$. Suppose for contradiction that $n-1\in \mathrm{set}(\beta)$. By Lemma~\ref{Basic}, $n$ is even-connected to itself with respect to the product $u_2\cdots u_s$, so there exists a closed walk $y_1-y_2-\cdots-y_{2\ell+2}$ with $y_1=y_{2\ell+2}=n$. Since $G$ is a tree (hence acyclic), every edge in this closed walk must appear an even number of times, which implies the walk has even length. But the walk has length $2\ell+1$ (odd), a contradiction. Thus $i=n-1\notin \mathrm{set}(\beta)$.

(2) If $j\in \mathrm{set}(\beta)$, then by Lemma~\ref{Basic}, there exists $k'\in [n]$ with $k'>j$ such that either $\{i,k'\}\in E(G)$ or $i$ is even-connected to $k'$ with respect to $u_2\cdots u_s$. By the vertex labeling in Remark \ref{labeling}, $j$ is the unique neighbor of $i$ satisfying $j>i$. Given $k'>j$, we have $\{i,k'\}\notin E(G)$. Hence, $i$ must be even-connected to $k'$ via a walk $y_1-y_2-\cdots-y_{2\ell+2}$ with $y_1=i$ and $y_{2\ell+2}=k'$. Let $p$ be the largest integer such that $y_p=i$.

If $p\neq 1$, then $y_1-y_2-\cdots-y_p$ is a closed walk, so $p$ is odd. It follows that $x_{y_{p+1}}x_{y_{p+2}}\in \{u_2,\ldots,u_s\}$, and Lemma~\ref{new} gives $y_{p+1}=j$.

- If $y_{p+2} > j$, then $x_{y_{p+1}}x_{y_{p+2}}=x_jx_{y_{p+2}}$ is the desired edge.

- If $y_{p+2}<j$, let $q$ denote the maximal integer such that $y_q=j$ and $q\ge p+1$. As all closed walks in $G$ have even length, $q-(p+1)$ is even; moreover, $y_{q+1}>j$ by Lemma~\ref{new}. The fact that $q-(p+1)$ is even implies $x_{y_q}x_{y_{q+1}}\in\{u_2,\dots,u_s\}$. Thus, $x_{y_q}x_{y_{q+1}}=x_jx_{y_{q+1}}$ is the required edge.

(3) We first note $n\notin \mathrm{set}(\beta)$ by Lemma~\ref{Basic}. It remains to show $n-1\notin \mathrm{set}(\beta)$. Assume for contradiction that $n-1\in \mathrm{set}(\beta)$. By part (2), there exists $t\in [s]$ such that $u_t=x_{n-1}x_n$. But by part (1), since $u_t$ has $j=n$, we get $n-1\notin \mathrm{set}(\beta)$, a contradiction. Hence $\mathrm{set}(\beta)\subseteq [n-2]$.
\end{proof}

Let $G$ be a graph on $[n]$, and $e=\{i,j\}$ be an edge of $G$, we set $\min(e)=\min\{i,j\}$.
\begin{Theorem} \label{TreeMain} Let $G$ be a tree with $n$ vertices. Then
\[
\mathrm{pd}(I_c(G)^s)=\left\{
\begin{array}{ll}
 s, & \hbox{ \text{if\ } $s\in  [n-2]$;} \\
  n-2, & \hbox{ \text{if\ } $s\geq n-2$.}
  \end{array}
  \right.
\]
\end{Theorem}
\begin{proof}
Let \( e_1,e_2,\ldots, e_{n-2}, e_{n-1} \) be all the  edges of \( G \), where \( e_{n-1}=\{n-1,n\}=x_{n-1}x_n \). Then, \( \min(e_1), \min(e_2), \ldots, \min(e_{n-2}) \) are pairwise distinct. In fact, if $\min(e_i)=\min(e_j):=k$, then $e_i=e_j=\{k,\ell\}$, where $\ell$ is the unique parent of $k$, see Remark~\ref{labeling}.

First, consider \( s \in [n-2] \). Choose \( \beta_s = \frac{\alpha^s}{e_1e_2\cdots e_s} \). It is clear that \( \beta_s \in \mathcal{G}(I_c(G)^s) \). By Lemma~\ref{Tree} (1), we directly deduce that
    \[
    \{\min(e_1),\min(e_2),\ldots, \min(e_s)\} \subseteq \operatorname{set}(\beta_s),
    \]
    which implies \( \operatorname{pd}(I_c(G)^s) \geq s \) by  formula~(\ref{eq:HerTak2}).

    Next, take an arbitrary \( \beta \in \mathcal{G}(I_c(G)^s) \). Then \( \beta \) can be written as \( \beta = \frac{\alpha^s}{e_{i_1}e_{i_2}\cdots e_{i_s}} \), where \( i_j \in [n-1] \) for each \( j \in [s] \). If \( i \in \operatorname{set}(\beta) \), then by Lemma~\ref{Tree} (2), there exists some \( t \in [s] \) such that \( i = \min(e_{i_t}) \). Combining this with Lemma~\ref{Tree} (1) and (3),  we immediately deduce that
    \begin{equation} \label{Tree1}
    \operatorname{set}(\beta) = \{\min(e_{i_j}) \mid j \in [s]\} \setminus \{n-1\}.
    \end{equation}
Therefore, \( |\operatorname{set}(\beta)| \leq s \). Since \( \beta \) is arbitrary, substituting into formula~(\ref{eq:HerTak2}) gives
    \[
    \operatorname{pd}(I_c(G)^s) = \max\left\{|\operatorname{set}(\beta)| : \beta \in \mathcal{G}(I_c(G)^s)\right\} \leq s.
    \]
    Combining the previously established lower bound (\( \operatorname{pd}(I_c(G)^s) \geq s \)) and the upper bound above, we conclude \( \operatorname{pd}(I_c(G)^s) = s \) for \( s \in [n-2] \).

Now, consider \( s \geq n-2 \). By Lemma~\ref{pdnondecreasing}, we have \( \operatorname{pd}(I_c(G)^{n-2}) \leq \operatorname{pd}(I_c(G)^s) \). On the other hand, we have already proven that \( \operatorname{pd}(I_c(G)^{n-2}) = n-2 \), and Lemma~\ref{Tree} (3) further implies \( \operatorname{pd}(I_c(G)^s) \leq n-2 \) for all $s\geq 1$. Thus,
    \[
    n-2 = \operatorname{pd}(I_c(G)^{n-2}) \leq \operatorname{pd}(I_c(G)^s) \leq n-2,
    \]
    which forces \( \operatorname{pd}(I_c(G)^s) = n-2 \). This completes the proof.
\end{proof}

This result generalizes \cite[Proposition 4.6]{FM3}, which proves that the same formula holds for paths, a specific type of trees.

We next determine the projective dimension of \( I_c(G) \) for all connected graphs \( G \).

\begin{Proposition}\label{non}
Let \( G \) be a connected graph. Then \( \mathrm{pd}(I_c(G)) \) is either \( 1 \) or \( 2 \). Moreover, \( \mathrm{pd}(I_c(G)) = 1 \) if and only if \( G \) is a tree.
\end{Proposition}

\begin{proof}
It is straightforward to see that \( 1 \leq \mathrm{pd}(I_c(G)) \leq 2 \) by the first assertion of Lemma~\ref{Basic}.

If \( G \) is a tree, then \( \mathrm{pd}(I_c(G)) = 1 \) by Theorem~\ref{TreeMain}. If \( G \) is not a tree, there exists a cycle \( C \) that is a subgraph of \( G \). By Theorem~\ref{Cycle} and Theorem~\ref{HSsubset}, we have \( \mathrm{pd}(I_c(G)) \geq \mathrm{pd}(I_c(C)) = 2 \).
\end{proof}

For any connected graph $G$, it is known from  Lemma~\ref{HSsubset} that $\mathrm{pd}(I_c(G)^s)$ is a non-decreasing function of $s$. This result can be strengthened. To do so, we first recall the definitions of a \emph{spanning subgraph} and a \emph{bipartite graph}: a subgraph \( H \) of \( G \) is spanning if they share the same vertex set, i.e., \( V(H) = V(G) \); a bipartite graph is a graph whose vertex set can be partitioned into two disjoint non-empty subsets \( X \) and \( Y \) such that every edge connects a vertex in \( X \) to a vertex in \( Y \). 

\begin{Proposition}\label{strict}
Let $G$ be a connected graph with $n$ vertices. Then  $\mathrm{pd}(I_c(G)^s)$ strictly increases until it reaches the value $n-2$.

In particular, if we furhter assume that $G$ is bipartite, then $\mathrm{pd}(I_c(G)^s)$ strictly increases until it reaches the maximal value.

\end{Proposition}
\begin{proof}
Let $T$ be a spanning tree  of $G$, and label its vertices  as described   in Remark \ref{labeling}. Note that this  also defines a labeling on $V(G)$ that satisfies the FM condition.  Suppose that $\mathrm{pd}(I_c(G)^s)=k\leq n-3$. We can choose  $\beta=\frac{\alpha^s}{u_1u_2\cdots u_s}\in \mathcal{G}(I_c(G)^s)$ such that $|\set(\beta)|=k$.
Let $\ell$ be the minimal element of $[n]\setminus \set(\beta)$.  Since $k\leq n-3$, it follows that $\ell\leq n-2$. Thus, there exists an integer $m$ with $\ell<m\leq n-1$ such that $\{\ell,m\}\in E(T)$.
Set $u_{s+1}=x_{\ell}x_{m}$ and $\gamma=\beta\cdot \frac{\alpha}{u_{s+1}}$. By Lemma~\ref{Basic},  $\ell\in \set(\gamma)$. Therefore,  $|\set(\gamma)|\geq k+1$. It follows that $\mathrm{pd}(I_c(G)^{s+1})\geq k+1$, which completes the proof of the first assertion.

For the second assertion, we only need to note that if $G$ is a connected bipartite graph, then  $\max\{\mathrm{pd}(I_c(G)^s)|s\geq 1\}=n-2$, by   \cite[Proposition 4.3]{FM3} and the Auslander-Buchsbaum formula.
\end{proof}

If \( G \) is a connected non-bipartite graph with \( n \) vertices, then the maximal value of \( \mathrm{pd}(I_c(G)^s) \) is \( n-1 \). However, we have not yet established whether it continues to increase strictly upon reaching the value \( n-2 \).

By \cite{B79a}, we know that
$\mathrm{depth}(S/I^k)=\mathrm{depth}(S/I^{k+1})$ for all $k\gg 0$. The least integer $k_0>0$ for which
$\mathrm{depth}(S/I^k)=\mathrm{depth}(S/I^{k_0})$
for all $k\ge k_0$ is called the \emph{index of depth stability} of $I$
and is denoted by $\mathrm{dstab}(I)$.

\begin{Corollary}\label{dstab}
Let $G$ be a connected bipartite graph with $n$ vertices. Then
\[
\mathrm{dstab} (I_c(G))\leq n-2.
\]
The equality holds if and only if $G$ is a tree.
\end{Corollary}
\begin{proof}
Combining Theorem~\ref{TreeMain} with the Auslander-Buchsbaum formula, we conclude that if \(G\) is a tree, then \(\mathrm{dstab}(I_c(G)) = n - 2\).

Suppose \(G\) contains at least one cycle. It suffices to show that \(\mathrm{dstab}(I_c(G)) < n - 2\). By Proposition~\ref{non}, we have \(\mathrm{pd}(I_c(G)) = 2\). Assume for contradiction that \(\mathrm{dstab}(I_c(G)) \geq n - 2\). Then, by Proposition~\ref{strict}, the following chain of inequalities holds:
\[
\mathrm{pd}(I_c(G)^{n-2}) \geq \mathrm{pd}(I_c(G)^{n-3}) + 1 \geq \cdots \geq \mathrm{pd}(I_c(G)) + (n - 3) = n - 1.
\]
This is a contradiction, which completes the proof.
\end{proof}
We now consider connected non-bipartite graphs. First, recall that a \emph{unicyclic graph} is a connected graph containing precisely one cycle. Note that every connected non-bipartite graph admits a unicyclic spanning subgraph that contains a unique odd cycle; we refer to such a subgraph as an \emph{odd unicyclic graph}.

Let \(H\) be an odd unicyclic graph with $n$ vertices, where \(C\) is the unique cycle of \(H\). Let \(k = |V(C)|\) denote the number of vertices on \(C\). We give a labeling of $V(H)$ that satisfies the FM condition.

 Step 1: Define \emph{the distance from a vertex to cycle \(C\)}. For any vertex \(v\in V(H)\), the distance from \(v\) to \(C\), denoted by \(d(v,C)\), is defined recursively as follows:

(1)  If \(v\in V(C)\), then \(d(v,C)=0\) (vertices on the cycle are at distance 0 from \(C\));

(2) If \(v\notin V(C)\), then \(d(v,C)=1+\min\left\{d(u,C)\mid u\in N(v)\right\}\), where \(N(v)\) is the set of neighbors of \(v\) in \(H\).

 Step 2: Label vertices in descending order of their distance to \(C\).
 Assign the \(k\) largest labels in \([n]\), namely \(n,n-1,\dots,n-k+1\), to the vertices of \(C\) following a consecutive cyclic order. Then for \(t=1,2,\dots\), assign the largest remaining unassigned labels arbitrarily to the vertices in the \(t\)-th distance layer, proceeding sequentially to the \((t+1)\)-th distance layer only after the \(t\)-th layer is fully labeled. Repeat this process until all vertices of \(H\) are labeled.

 \begin{Lemma} \label{way} Let $H$  be an odd unicyclic graph, where \(C\) is the unique cycle of \(H\). Label $V(H)$ as the above way. Then
  this labelling satisfies the FM condition: for any \(i\in[n-1]\), the induced subgraph  of $H$  on $\{i,i+1,\dots,n\}$ is connected.
  \end{Lemma}
  \begin{proof} Fix $i\in [n-1]$. By induction, we may assume that the induced subgraph of $H$ on $[i+1,n]$ is connected. To prove that the induced subgraph of $H$ on $[i,n]$ is connected, it suffices to prove that there exists $i<j\leq n$ such that $i$ is adjacent to $j$. We consider the following two cases.

  If \(i\geq n-k+1\), then \(i\in V(C)\). By the consecutive cyclic labelling rule for \(C\) and the fact that \(i\leq n-1\), vertex \(i\) is adjacent to vertex \(i+1\). Thus, we are done.

If \(i<n-k+1\), then \(i\notin V(C)\) and so \(d(i,C)\geq 1\). By the definition of the distance function, there exists a neighbor \(j\) of vertex \(i\) such that \(d(j,C)=d(i,C)-1\). By the labelling rule, it follows that \(j>i\). This completes the proof.
\end{proof}

\begin{Theorem} \label{uni2}
Let $G$ be a connected non-bipartite graph with $n$ vertices. Then
\[
\mathrm{dstab} (I_c(G))\leq n-2.
\]
\end{Theorem}

\begin{proof}
Let $H$ be a spanning subgraph of $G$ that is odd unicyclic, and let $C$ be the unique cycle of $H$. Set $k:=|V(C)|$. By the preceding discussion, there exists a labeling of $V(H)$ by $[n]$ that satisfies the FM condition, where the vertices of $C$ are labeled by $n,n-1,\ldots,n-k+1$ in a consecutive cyclic order. Since $H$ is a spanning subgraph of $G$, this labeling also serves as a labeling of $G$ satisfying the FM condition, so Lemma~\ref{Basic} is applicable here.

Recall that an odd unicyclic graph has exactly as many edges as vertices (i.e., $|E(H)|=|V(H)|=n$), as it is a tree plus one additional edge forming a cycle. Let $e_1,e_2,\ldots,e_n$ denote all edges of $H$. Without loss of generality, we assume $e_n=\{n-1,n\}$ and $e_{n-1}=\{n-k+1,n\}$. Define
\[
\beta:= \frac{\alpha^{n-2}}{e_1e_2\cdots e_{n-3}} \in I_c(G)^{n-2}.
\]
We claim that $\mathrm{set}(\beta) = [n-1]$. To verify this claim, we proceed by cases for $p \in [n-1]$:

\noindent \textbf{Case 1: $p = n-1$} \\
There exists some index $1\leq i\leq n-3$ (without loss of generality, take $i=1$) such that $e_1=\{n-2,n-1\}$. If $k=3$, then $n-2 = n-k+1$, so $n-2$ is adjacent to $n$. For $k>3$, $n-2$ is even-connected to $n$ with respect to the product $e_2\cdots e_{n-3} $ via the walk: $(n-2)-(n-3)-(n-4)-\cdots-1-n$. By Lemma~\ref{Basic}, it follows that $p=n-1 \in \mathrm{set}(\beta)$.

\noindent \textbf{Case 2: $n-k+1\leq p < n-1$} \\
Since $p\in V(C)$ (by the labeling rule for $C$), we may assume $e_1=\{p,p+1\}$. Note that $p+1$ is adjacent to $p+2$, so by Lemma~\ref{Basic}, $p\in \mathrm{set}(\beta)$.

\noindent \textbf{Case 3: $1\leq p < n-k+1$} \\
Here, $p\notin V(C)$, so by the FM condition for the labeling of $H$, there exists a vertex $q>p$ such that $\{p,q\} \in E(H)$. Suppose first that $q=n$. Set $e_1 = \{p,q\}=\{p,n\}$.  Under this edge designation, since $p\in V(C)$,   $e_1 \in E(H)\setminus E(C)$. Hence, $n$ is even-connected to itself with respect to the product $e_2\cdots e_{n-3}$. If $q<n$, then by the FM condition (applied to vertex $q$), there exists a vertex $r>q$ such that $\{q,r\} \in E(H)$. In either case, Lemma~\ref{Basic} implies $p\in \mathrm{set}(\beta)$.

This completes the proof of the claim that $\mathrm{set}(\beta) = [n-1]$.

Since $\mathrm{set}(\beta) = [n-1]$, we have $\mathrm{pd}(I_c(G)^{n-2})=n-1$. Recall that the projective dimension $\mathrm{pd}(I_c(G)^s)$ is a non-decreasing function of $s\geq 1$, and its maximal possible value is $n-1$. Thus, $\mathrm{pd}(I_c(G)^s)=n-1$ for all $s\geq n-2$, which implies $\mathrm{dstab}(I_c(G))\leq n-2$. This completes the proof of the theorem.
\end{proof}

\section{Homological shift algebra}\label{sec:homological}

In this section, we first show that Conjecture~\ref{conjecture} holds for the complementary edge ideal \( I \) of any connected graph when \( i = 1 \). We then establish explicit closed-form expressions for \( \mathrm{HS}_i(I^s) \) corresponding to the complementary edge ideals \( I \) of trees and cycles. These expressions, in turn, provide a direct proof that Conjecture~\ref{conjecture} holds for these classes of ideals.
We use $\alpha$ to denote ${\bf x}_{[n]}$ in this section.

\subsection{The case when $i=1$}
Recall from \cite[Lemma 2.3]{FQ1} that if \( I \) is a monomial ideal (regardless of whether it has linear quotients), then \( \mathrm{HS}_1(I) \) is generated by all \( \mathrm{lcm}(u, v) \), where \( u, v \in \mathcal{G}(I) \) and \( u \neq v \). Here, \( \mathrm{lcm}(u, v) \) denotes the least common multiple of \( u \) and \( v \).

\begin{Theorem}\label{HS1generated}
Let $G$ be a connected graph on the set  $[n]$ and  let $I=I_c(G)$ be its complementary edge ideal. Then $\HS_1(\mathcal{R}(I))$ is generated in degree $1$ as an $\mathcal{R}(I)$-module.
\end{Theorem}

\begin{proof}
It suffices to prove that \( \mathrm{HS}_1(I^{s+1}) = I^s \cdot \mathrm{HS}_1(I) \) for all \( s \geq 0 \).

First, we recall two  definitions for clarity:

- For any vertex \( i \in V(G) \), \( \deg_G(i) \) denotes  the number of edges incident to \( i \).

- For two monomials \( u, v \in S \), \( \gcd(u, v) \) stands for the greatest common divisor of \( u \) and \( v \).

 By \cite[Lemma 2.3]{FQ1} and the fact that \( \mathrm{HS}_1(I)\) is generated in degree $n-1$, it is straightforward to verify that \( \mathrm{HS}_1(I) = \left( \frac{\alpha}{x_i} \mid \deg_G(i) \geq 2 \right) \). We now proceed by induction on \( s \). Note that \( I^0 = S \), so the case when \( s = 0 \) is trivial.

Suppose \( s \geq 1 \). Let \( \beta \in \mathcal{G}(\mathrm{HS}_1(I^{s+1})) \). By \cite[Lemma 2.3]{FQ1}, there exist elements \( u, v \in \mathcal{G}(I(G)^{s+1}) \) such that \( \beta = \mathrm{lcm}\left(\frac{\alpha^{s+1}}{u}, \frac{\alpha^{s+1}}{v}\right)=\frac{\alpha^{s+1}}{\gcd(u, v)} \), where \( \deg(\gcd(u, v)) = 2s + 1 \). It follows that \( u = y \cdot \gcd(u, v) \) for some \( y \in \{x_1, \ldots, x_n\} \).

Express \( u \) as \( u = \prod_{j=1}^{s+1} u_j \), where each \( u_j \in \mathcal{G}(I(G)) \). Without loss of generality, assume \( y \) divides \( u_1 \), so \( u_1 = x_i y \) for some \( i \in [n] \). Thus, \( \gcd(u, v) = x_i \cdot \prod_{j=2}^{s+1} u_j \). We consider two cases:

(1) If \( \deg_G(i) \geq 2 \), then
\[
\beta = \frac{\alpha^s}{u_2 \cdots u_{s+1}} \cdot \frac{\alpha}{x_i} \in I^s \cdot \mathrm{HS}_1(I).
\]

(2) If \( \deg_G(i) = 1 \), express \( v \) as \( v = \prod_{j=1}^{s+1} v_j \), where each \( v_j\in \mathcal{G}(I(G)) \). Since \( x_i \mid \gcd(u, v) \), we may assume \( x_i \) divides \( v_1 \), which implies \( u_1 = v_1 \) (as both are generators of \( I(G) \) containing \( x_i \)). Let \( u_1 = v_1 = w \) for simplicity. Set \( f = \prod_{j=2}^{s+1} u_j \) and \( g = \prod_{j=2}^{s+1} v_j \). Then \( \gcd(u, v) = w \cdot \gcd(f, g) \), and \( \frac{\alpha^s}{\gcd(f, g)} \in \mathcal{G}(\mathrm{HS}_1(I^s)) \). By the induction, it follows that \( \frac{\alpha^s}{\gcd(f, g)} \in I^{s-1} \cdot \mathrm{HS}_1(I) \). Therefore:
\[
\beta = \frac{\alpha}{w} \cdot \frac{\alpha^s}{\gcd(f, g)} \in I^s \cdot \mathrm{HS}_1(I),
\]
as required.
\end{proof}

For $i\geq 2$,  there is no uniform method to determine the generating degree of the $i$-th homological shift algebra of the complementary edge ideal of a connected graph.

\subsection{Trees}

\begin{Theorem}\label{3.1} Let $G$ be a tree with $n$ vertices, labelled as in Remark \ref{labeling}, and let $I=I_c(G)$ denote its complementary edge ideal. Define a map $\phi: [n-1]\rightarrow [n]$ such that $\phi(k)$ is the unique $j\in [n]$ satisfying $j>k$ and $\{k,j\}\in E(G)$. Then, for any $1\leq i\leq s$, we have
\[
\mathrm{HS}_i(I^s)=I^{s-i}\cdot \left(\frac{\alpha^i}{\prod_{k\in F}x_{\phi(k)}}\middle|\; F\subseteq [n-2] \text{\ and\ } |F|=i \right).
\]
\end{Theorem}

\begin{proof} Let  $e_1,e_2,\ldots, e_{n-2}, e_{n-1}$ be all the edges of $G$. Since $\min(e_1), \ldots, \min(e_{n-1})$ are pairwise distinct, we may assume without loss of generality that $\min (e_j)=j$ for every $j\in [n-1]$.
 It follows that $e_j=x_jx_{\phi(j)}$ for all $j\in [n-1]$. By \cite[Lemmas 5.9-5.11]{HHO}, the edge ring $\mathbb{K}[x_jx_{\phi(j)}|1\leq j\leq n-1]$  has a trivial defining ideal. This implies that the quadratic monomials  $e_1,\ldots,e_{n-1}$ are algebraically independent over $\mathbb{K}$. Hence,   every element of $\mathcal{G}(I(G)^s)$ is uniquely written as
\[
e_1^{s_1}e_2^{s_2}\cdots e_{n-1}^{s_{n-1}},
\]
where $\sum\limits_{j=1}^{n-1}s_j=s$ and each $s_j\geq 0$.

Take any \( \beta = \frac{\alpha^{s-i}}{e_1^{a_1} \cdots e_{n-1}^{a_{n-1}}} \in \mathcal{G}(I^{s-i}) \), where \( e_1^{a_1} \cdots e_{n-1}^{a_{n-1}} \in \mathcal{G}(I(G)^{s-i}) \), whence \( \sum_{j=1}^{n-1} a_j = s-i \). For any subset \( F \subseteq [n-2] \) with \( |F| = i \), define
\[
\gamma=\frac{\alpha^s}{e_1^{a_1+b_1}\cdots e_{n-1}^{a_{n-1}+b_{n-1}}},
\]
where $b_i=1$ if $i\in F$, and $b_i=0$ otherwise.   In particular, $b_{n-1}=0$.  Note that $\gamma\in \mathcal{G}(I^{s})$, and $F\subseteq \set(\gamma)$ by  the equation~(\ref{Tree1}). This implies that ${\bf x}_F\gamma\in \mathrm{HS}_i(I^s)$. Direct computation gives
\[
 \beta \cdot \frac{\alpha^i}{\prod_{k \in F} x_{\phi(k)}}=\mathbf{x}_F \gamma\in \mathrm{HS}_i(I^s).
\]
Thus, the inclusion \( \supseteq \) holds.

Conversely, take any \( U \in \mathcal{G}(\mathrm{HS}_i(I^s)) \). By definition of \( \mathrm{HS}_i(I^s) \), there exists a subset \( F \subseteq [n] \) with \( |F| = i \) and \( V = \frac{\alpha^s}{e_1^{s_1} \cdots e_{n-1}^{s_{n-1}}} \in \mathcal{G}(I^s) \) such that \( U = \mathbf{x}_F V \) and \( F \subseteq \mathrm{set}(V) \). By Equation~(\ref{Tree1}), \( F \subseteq [n-2] \) and \( s_k > 0 \) for all \( k \in F \). Define \( c_k = s_k - 1 \) if \( k \in F \) and \( c_k = s_k \) otherwise; then \( \sum_{k=1}^{n-1} c_k = s - i \), so \( W = \frac{\alpha^{s-i}}{e_1^{c_1} \cdots e_{n-1}^{c_{n-1}}} \in \mathcal{G}(I^{s-i}) \). Direct computation verifies
\[
W \cdot \frac{\alpha^i}{\prod_{k \in F} x_{\phi(k)}} = \mathbf{x}_F V = U,
\]
so \( U \in I^{s-i} \cdot \left( \frac{\alpha^i}{\prod_{k \in F} x_{\phi(k)}} \middle|\; F \subseteq [n-2], |F| = i \right) \). Thus, the inclusion \( \subseteq \) holds.

Combining both inclusions completes the proof.
   \end{proof}

\begin{Corollary}
Let \( G \) be a tree and let \( I = I_c(G) \) be its complementary edge ideal. For \( 1 \leq i \leq n-2 \), the \( i \)-th homological shift algebra \( \mathrm{HS}_i(\mathcal{R}(I)) \) is generated exactly in degree \( i \).
\end{Corollary}

\begin{proof}
By Theorem~\ref{3.1}, we have \( \mathrm{HS}_i(I^s) = I^{s-i} \cdot \mathrm{HS}_i(I^i) \) for all \( s \geq i \). This inclusion implies \( \mathrm{HS}_i(\mathcal{R}(I)) \) is generated in degrees at most \( i \). By Theorem~\ref{TreeMain}, it follows that \( \mathrm{HS}_i(I^s) = 0 \) for all \( s < i \) and  no generators of \( \mathrm{HS}_i(\mathcal{R}(I)) \) can exist in degrees less than \( i \). Combining these two observations, this $\mathcal{R}(I)$-module is generated exactly in degree \( i \).
\end{proof}

The ideal \( \left( \frac{\alpha^i}{\prod_{k \in F} x_{\phi(k)}}\middle|\; F \subseteq [n-2] \text{ and } |F| = i \right) \) is uniquely determined by \( G \) and \( i \) up to variable permutation, as it coincides with \( \mathrm{HS}_i(I^i) \). This implies the ideal is independent of the labeling of \( G \), and we denote it by \( J_i(G) \).

To study the properties of \( J_i(G) \), we recall some foundational theory on polymatroidal ideals, following the notation from \cite{CMS} and \cite{HHBook}.

Given a monomial \( u = x_1^{c_1}x_2^{c_2}\cdots x_n^{c_n} \in S \) where each \( c_i \in \mathbb{Z}_{\ge 0} \), the \( x_i \)-degree of \( u \), denoted \( \deg_{x_i}(u) \), is defined as \( \deg_{x_i}(u) = c_i \). An ideal \( I \subset S \) generated by monomials of a single degree is called \textit{polymatroidal} if it satisfies the \textit{exchange property}, which states that for all \( u, v \in \mathcal{G}(I) \) and every index \( i \) with \( \deg_{x_i}(u) > \deg_{x_i}(v) \), there exists an index \( j \) such that \( \deg_{x_j}(u) < \deg_{x_j}(v) \) and \( x_j(u/x_i) \in I \).

A stronger version, known as the \textit{strong exchange property}, asserts that for all \( u, v \in \mathcal{G}(I) \), every index \( i \) with \( \deg_{x_i}(u) > \deg_{x_i}(v) \), and \(\textit{all}\) indices \( j \) with \( \deg_{x_j}(u) < \deg_{x_j}(v) \), the monomial \( x_j(u/x_i) \) lies in \( I \). It is immediate that the strong exchange property implies the exchange property; consequently, every ideal satisfying the strong exchange property is polymatroidal.

A key class of polymatroidal ideals that satisfy the strong exchange property are \textit{ideals of Veronese type} (also referred to as \textit{Veronese-type ideals}): given a vector \( {\bf a} = (a_1, \ldots, a_n) \in \mathbb{Z}_{\ge 0}^n \) and an integer \( d \) with \( |{\bf a}| = \sum_{i=1}^n a_i \ge d \), the ideal of Veronese type determined by \( {\bf a} \) and \( d \) is defined as
\[
I_{{\bf a},d} = \left( {\bf x}^b \in S : |{\bf b}| = d, \, {\bf b} \le {\bf a} \right),
\]
where \( {\bf b} = (b_1, \ldots, b_n) \in \mathbb{Z}_{\ge 0}^n \), \( |{\bf b}| = \sum_{i=1}^n b_i \) (the total degree of \( {\bf x}^b \)), and \( {\bf b} \le {\bf a} \) means \( b_i \le a_i \) for all \( i \). Ideals of Veronese type have been studied in \cite{FL, HMRZ021a, HHV, HRV}. By the work \cite{HHV}, it is known that ideals satisfying the strong exchange property are essentially ideals of Veronese type in the sense: For a monomial ideal $I\subset S$  generated in a single degree. Then $I$ satisfies the strong exchange property, if and only if, $I=(u)I_{{\bf a},d}$ for some monomial $u\in S$, and some ${\bf a}\in\ZZ_{\ge0}^n$ and $d>0$.

\begin{Proposition}\label{ith homology}
Let \( G \) be a tree with \( n \) vertices. Then \( \mathrm{HS}_i(I_c(G)^i) = J_i(G) \) is a polymatroidal ideal with the strong exchange property for each \( i \in [n-2] \).
\end{Proposition}
\begin{proof}
Let \(K_i(G)\) denote the monomial ideal of \(S\) generated by all monomials of the form \( \prod_{k\in F}x_{\phi(k)} \), where \(F\subseteq [n-2]\) and \(|F|=i\). For any $ k \in [n-1]$, define
\[
 b_k := |\phi^{-1}(k)| = \left| \{ j \in [n-2] : \phi(j) = k \} \right|.
  \]
Then $b_k > 0$  if and only if $k \in \phi([n-2])$. Moreover, $\sum\limits_{i=1}^{n-1} b_i = n-2$. We claim that
\[
K_i(G) = \left(\prod\limits_{k=1}^{n-1} x_k^{a_k} \middle|\; \sum\limits_{k=1}^{n-1} a_k = i\text{\ and\ } 0 \leq a_k \leq b_k\ \text{for all}\  k\in [n-1] \right).
 \]
Fix $i\in [n-2]$. For every $F\subseteq [n-2]$ such that $|F| = i$, set $a_k:= |F \cap \phi^{-1}(k)|$ for any $k\in [n-1]$. Then $0 \leq a_k \leq b_k$ for all $k$, and
\[
\sum\limits_{k=1}^{n-1} a_k = \left| \bigsqcup_{k=1}^{n-1} (F \cap \phi^{-1}(k)) \right| = |F| = i.
\]
Note that $\prod\limits_{k \in F} x_{\phi(k)} =\prod\limits_{k=1}^{n-1} x_k^{a_k}$, so the inclusion $\subseteq$ of the above claim holds.

Conversely, take any monomial $\prod\limits_{k=1}^{n-1} x_k^{a_k}$ such that $\sum\limits_{k=1}^{n-1} a_k= i$ and $0 \leq a_k \leq b_k$ for all $k\in [n-1]$.   Let $F = F_1 \sqcup \cdots \sqcup F_{n-1}$, where  $F_k$ is a subset of $a_k$ elements in $\phi^{-1}(k)$,  and by convention, $F_k=\emptyset$ if $a_k = 0$.
Then $\prod_{k \in F} x_{\phi(k)} =\prod\limits_{k=1}^{n-1} x_k^{a_k}$, so the inclusion
$\supseteq$ holds. This proves the claim. Therefore,  $K_i(G)$ coincides with  the Veronese-type ideal $I_{\mathbf{b},i}$, where ${\bf b}=(b_1,\ldots,b_{n-1}).$

By putting $b_n=0$, we  conclude that:
\begin{align*}
&J_i(G) = \left( \prod\limits_{k=1}^{n} x_k^{i-a_k} \middle|\; \sum\limits_{k=1}^{n-1} a_k = i\text{\ and\ } 0 \leq a_k \leq b_k\ \text{for all}\  k\in [n] \right) \\
&=\left(\prod\limits_{k=1}^{n} x_k^{c_k} \middle|\; \sum\limits_{k=1}^{n-1}c_k= (n-2)i\text{\ and\ } i - b_k \leq c_k \leq i\ \text{for all}\  k\in [n] \right) \\
&=\left(\prod\limits_{k=1}^{n} x_k^{i-b_k}\right)\left(\prod\limits_{k=1}^{n-1} x_k^{c_k-i+b_k} \middle|\; \sum\limits_{k=1}^{n-1}c_k= (n-2)i \text{\ and\ } i - b_k \leq c_k \leq i\ \text{for all}\  k\in [n] \right)\
\end{align*}
\begin{align*}
&=\left(\prod\limits_{k=1}^{n} x_k^{i-b_k}\right)\left(\prod\limits_{k=1}^{n-1} x_k^{a_k} \middle|\; \sum\limits_{k=1}^{n-1}a_k= n-2-i\text{\ and\ } 0 \leq a_k \leq b_k\ \text{for all}\  k\in [n] \right)\\&
= \left(\prod\limits_{k=1}^{n} x_k^{i-b_k}\right)K_{n-2-i}(G).
\end{align*}
Here, the  second to last equality holds because of $\sum_{k=1}^{n-1} b_k = n-2$.
In the last equality, we use the convention that $K_0(G)=S$.
Therefore,  $J_i(G)$ is a polymatroidal ideal with the strong exchange property.
 \end{proof}
Let \(I\) be a polymatroidal ideal. A \textit{cage} for \(I\) is an integral vector \(\mathbf{c}\in\mathbb{N}^n\) such that \(\mathbf{u}(i)\le\mathbf{c}(i)\) for all \(\mathbf{u}\) with \(\mathbf{x}^{\mathbf{u}}\in\mathcal{G}(I)\) and all \(i\in[n]\). Given a cage \(\mathbf{c}\) of \(I\), the ideal generated by \(\mathbf{x}^{\mathbf{c}-\mathbf{u}}\) for all \(\mathbf{u}\in\mathcal{G}(I)\) is itself a polymatroidal ideal; we refer to this object as the \textit{dual polymatroidal ideal of \(I\) with respect to \(\mathbf{c}\)}. For more details, see \cite{CMS, FL}. By virtue of this definition, \(J_i(G)\) is precisely the dual polymatroidal ideal of \(K_i(G)\) with respect to the \(n\)-tuple cage \(\mathbf{c}=(i,\dots,i)\).

In Proposition~\ref{ith homology}, we show that the $i$-th homological shift ideal $\mathrm{HS}_i(I_c(G)^i)$ (denoted $J_i(G)$) is a Veronese-type ideal after dividing by a suitable monomial. In fact, every Veronese-type ideal arises in this manner.

Recall that a \textit{caterpillar tree} is a tree $T$ containing a path $P$ such that every vertex of $T$ is either on $P$ or adjacent to a vertex of $P$. Caterpillar trees were first investigated by Harary and Schwenk \cite{HS1973}. A \textit{leaf edge} is an edge connecting a leaf vertex (a vertex of degree 1) to another vertex in the graph.

\begin{Theorem}\label{AllVeronese}
Given a vector $\mathbf{b} = (b_1, \ldots, b_n) \in \mathbb{Z}_{>0}^n$ and a positive integer $d \leq |\mathbf{b}|$, where $|\mathbf{b}| = \sum_{j=1}^n b_j$. Set $\sigma_k = \sum_{j=1}^k b_j$ for all $k \in [n]$, and let $i = |\mathbf{b}| - d$. Then there exists a caterpillar tree $T$ with $|\mathbf{b}| + 2$ vertices such that
\[
\mathrm{HS}_{i}\bigl(I_c(T)^{i}\bigr) = uI,
\]
where $u$ is a monomial, and $I \subseteq \mathbb{K}[x_j \mid 1 \leq j \leq |\mathbf{b}| + 2]$ is a Veronese-type ideal generated by all monomials of the form $x_{\sigma_1+1}^{a_1}x_{\sigma_2+1}^{a_2}\cdots x_{\sigma_n+1}^{a_n}$ with $0 \leq a_j \leq b_j$ for all $j \in [n]$ and $\sum_{t=1}^n a_t = d$.
\end{Theorem}

We remark that the ideal $I$ is essentially identical to the ideal $I_{\mathbf{b}, d}\subset S$, despite their residing in the distinct  ambient rings.
\begin{proof}
 Recall from the proof of Proposition~\ref{ith homology} that for a tree $G$ with $n$ vertices, we use $K_i(G)$ to denote the dual polymatroidal ideal of $\mathrm{HS}_i(I_c(G)^i)$ with respect to the $n$-tupe $(i,\ldots,i)$.
 We now construct a caterpillar tree $T$ on vertex set $[\sigma_n+2]$ such that $K_d(T) = I$.

 At first, let $P$ be a path of length $n$ with $n+1$ vertices, labeled in order by $\sigma_1 + 1, \sigma_2 + 1, \ldots, \sigma_n + 1, \sigma_n + 2$. We then add leaf edges according to the sequence $(b_1, \ldots, b_n)$: attach $\sigma_1$ leaf edges to the vertex $\sigma_1 + 1$, whose leaf vertices are $1, \ldots, \sigma_1$; for each $2 \leq i \leq n$, attach $\sigma_i - \sigma_{i-1} - 1 = b_i - 1$ leaf edges to the vertex $\sigma_i + 1$, with leaf vertices $\sigma_{i-1} + 2, \ldots, \sigma_{i-1} + b_i = \sigma_i$.

Denote the resulting graph by $T$. By construction, $T$ is a caterpillar tree with vertex set $[\sigma_n + 2]$, and  its vertex labeling satisfies the FM condition. Define  $\phi: [\sigma_n + 1] \to [\sigma_n + 2]$ by assigning to each $k\in [\sigma_n + 1]$ the unique vertex $j\in [\sigma_n + 2]$ such that $j > k$ and $\{j, k\} \in E(T)$. For the role of $\phi$, see Theorem~\ref{3.1}.

\begin{figure}[htbp]
	\begin{center}
		\begin{tikzpicture}[thick, scale=0.9, every node/.style={scale=0.85}]
			
			\draw[solid] (-1,5)--(1.5,5)--(4.0,5);
			\shade [shading=ball, ball color=black] (-1,5) circle (.07) node[below] {\scriptsize$\sigma_1+1$};
			\shade [shading=ball, ball color=black] (1.5,5) circle (.07) node[above] {\scriptsize$\sigma_2+1$};
			\shade [shading=ball, ball color=black] (4.0,5) circle (.07) node[below] {\scriptsize$\sigma_3+1$};
			
			\draw[solid] (-1,5)--(-2.5,6);
			\draw[solid] (-1,5)--(0.5,6);
			\draw[solid] (-1,5)--(-1.6,6);
			\draw[solid] (-1,5)--(-0.4,6);
			\shade [shading=ball, ball color=black] (-2.5,6) circle (.06) node[above] {\scriptsize$1$};
			\shade [shading=ball, ball color=black] (-1.6,6) circle (.06) node[above] {\scriptsize$2$};
			\shade [shading=ball, ball color=black] (-0.4,6) circle (.06) node[above] {\scriptsize$\sigma_1-1$};
			\shade [shading=ball, ball color=black] (0.5,6) circle (.06) node[above] {\scriptsize$\sigma_1$};
			\shade [shading=ball, ball color=black] (-1.2,5.8) circle (.03);
			\shade [shading=ball, ball color=black] (-1.0,5.8) circle (.03);
			\shade [shading=ball, ball color=black] (-0.8,5.8) circle (.03);
			
			\draw[solid] (1.5,5)--(0,4);
			\draw[solid] (1.5,5)--(3,4);
			\draw[solid] (1.5,5)--(0.9,4);
			\draw[solid] (1.5,5)--(2.1,4);
			\shade [shading=ball, ball color=black] (0,4) circle (.06) node[below] {\scriptsize$\sigma_1+2$};
			\shade [shading=ball, ball color=black] (0.9,4) circle (.06) node[below] {\scriptsize$\sigma_1+3$};
			\shade [shading=ball, ball color=black] (2.1,4) circle (.06) node[below] {\scriptsize$\sigma_2-1$};
			\shade [shading=ball, ball color=black] (3,4) circle (.06) node[below] {\scriptsize$\sigma_2$};
			\shade [shading=ball, ball color=black] (1.3,4.2) circle (.03);
			\shade [shading=ball, ball color=black] (1.5,4.2) circle (.03);
			\shade [shading=ball, ball color=black] (1.7,4.2) circle (.03);
			
			\draw[solid] (4.0,5)--(2.5,6);
			\draw[solid] (4.0,5)--(5.5,6);
			\draw[solid] (4.0,5)--(3.4,6);
			\draw[solid] (4.0,5)--(4.6,6);
			\shade [shading=ball, ball color=black] (2.5,6) circle (.06) node[above] {\scriptsize$\sigma_2+2$};
			\shade [shading=ball, ball color=black] (3.4,6) circle (.06) node[above] {\scriptsize$\sigma_2+3$};
			\shade [shading=ball, ball color=black] (4.6,6) circle (.06) node[above] {\scriptsize$\sigma_3-1$};
			\shade [shading=ball, ball color=black] (5.5,6) circle (.06) node[above] {\scriptsize$\sigma_3$};
			\shade [shading=ball, ball color=black] (3.8,5.8) circle (.03);
			\shade [shading=ball, ball color=black] (4.0,5.8) circle (.03);
			\shade [shading=ball, ball color=black] (4.2,5.8) circle (.03);
			
			\shade [shading=ball, ball color=black] (5.0,5) circle (.02);
			\shade [shading=ball, ball color=black] (5.3,5) circle (.02);
			\shade [shading=ball, ball color=black] (5.6,5) circle (.02);
			\shade [shading=ball, ball color=black] (5.9,5) circle (.02);
			\shade [shading=ball, ball color=black] (6.2,5) circle (.02);
			\shade [shading=ball, ball color=black] (6.5,5) circle (.02);
			
			\draw[solid] (7.5,5)--(10.0,5)--(12.5,5);
			\shade [shading=ball, ball color=black] (7.5,5) circle (.07) node[above] {\scriptsize$\sigma_{n-1}+1$};
			\shade [shading=ball, ball color=black] (10.0,5) circle (.07) node[below] {\scriptsize$\sigma_n+1$};
			\shade [shading=ball, ball color=black] (12.5,5) circle (.07) node[below] {\scriptsize$\sigma_n+2$};
			
			\draw[solid] (7.5,5)--(5.5,4);
			\draw[solid] (7.5,5)--(9.5,4);
			\draw[solid] (7.5,5)--(6.8,4);
			\draw[solid] (7.5,5)--(8.2,4);
			\shade [shading=ball, ball color=black] (5.5,4) circle (.06) node[below] {\scriptsize$\sigma_{n-2}+2$};
			\shade [shading=ball, ball color=black] (6.8,4) circle (.06) node[below] {\scriptsize$\sigma_{n-2}+3$};
			\shade [shading=ball, ball color=black] (8.2,4) circle (.06) node[below] {\scriptsize$\sigma_{n-1}-1$};
			\shade [shading=ball, ball color=black] (9.5,4) circle (.06) node[below] {\scriptsize$\sigma_{n-1}$};
			\shade [shading=ball, ball color=black] (7.3,4.2) circle (.03);
			\shade [shading=ball, ball color=black] (7.5,4.2) circle (.03);
			\shade [shading=ball, ball color=black] (7.7,4.2) circle (.03);
			
			\draw[solid] (10.0,5)--(8.0,6);
			\draw[solid] (10.0,5)--(12.0,6);
			\draw[solid] (10.0,5)--(9.3,6);
			\draw[solid] (10.0,5)--(10.7,6);
			\shade [shading=ball, ball color=black] (8.0,6) circle (.06) node[above] {\scriptsize$\sigma_{n-1}+2$};
			\shade [shading=ball, ball color=black] (9.3,6) circle (.06) node[above] {\scriptsize$\sigma_{n-1}+3$};
			\shade [shading=ball, ball color=black] (10.7,6) circle (.06) node[above] {\scriptsize$\sigma_n-1$};
			\shade [shading=ball, ball color=black] (12.0,6) circle (.06) node[above] {\scriptsize$\sigma_n$};
			\shade [shading=ball, ball color=black] (9.8,5.8) circle (.03);
			\shade [shading=ball, ball color=black] (10.0,5.8) circle (.03);
			\shade [shading=ball, ball color=black] (10.2,5.8) circle (.03);
			
		\end{tikzpicture}
		\caption{Caterpillar tree $T$}
		\label{fig3}
	\end{center}
\end{figure}
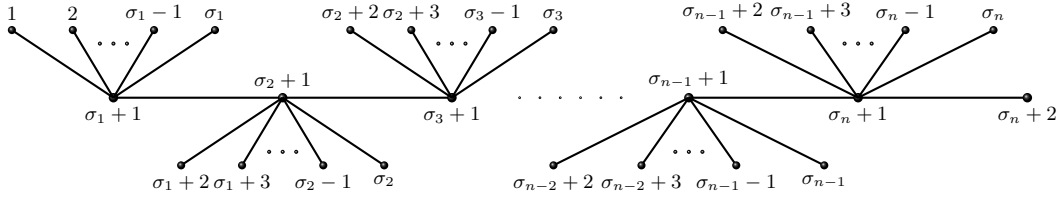It is straightforward to verify that  for any $k \in [\sigma_n+1]$, the preimage $\phi^{-1}(k)$ is nonempty if and only if $k = \sigma_j + 1$ for some $1 \le j \le n$.
Furthermore, for each $j\in  [n]$, we have
$\phi^{-1}(\sigma_j + 1) = \{\sigma_{j-1}+1, \sigma_{j-1}+2, \ldots, \sigma_j\}$, whence $|\phi^{-1}(\sigma_j + 1)|=b_j$.  By the proof of Proposition~\ref{ith homology}, it follows that $K_d(T)=I$. Hence, we conclude $$\mathrm{HS}_{|\mathbf{b}|-d}(I_c(T)^{|\mathbf{b}|-d})=J_{|\mathbf{b}|-d}(T)=uK_d(T)=uI$$ for some monomial $u$, as required.
\end{proof}
\subsection{Cycles}

\medskip

In  this subsection, we investigate the homological shift algebras associated to the complementary edge ideal of a cycle.

Let \( G \) be an \( n \)-cycle with vertices labeled \( 1, 2, \ldots, n \) in a consecutive cyclic order. For \( j = 1, \ldots, n-1 \), let \( {\bf e}_j \) denote the edge \( x_jx_{j+1} \), and set \( {\bf e}_n = x_nx_1 \). For notational convenience, we also use \( {\bf e}_0 \) to denote \( {\bf e}_n \) and \( x_0 \) to denote \( x_n \).

We first consider the case when $G$ is an even cycle. By the definition of a bipartite graph, it is easy to see that  every closed walk in a bipartite graph has even length.

\begin{Lemma}\label{even} Let $G$ be the even cycle of length $n=2m$ and let $I$ be the complementary edge ideal of $G$.  Take $\beta=\frac{\alpha^s}{u_1u_2\cdots u_s}\in I^s$, where $u_j\in I(G)$ for all $j$. Then  $\set(\beta)\subseteq [n-2]$, and the following assertions hold:

\begin{enumerate}
\item If $u_1={\bf e}_{2\ell-1}=x_{2\ell-1}x_{2\ell}$ for some $1\leq \ell\leq m-1$,   then $2\ell$ belongs to $\set(\beta)$ if and only if either $u_j=x_{2\ell}x_{2\ell+1}$ for some $j\geq 2$ or $\{{\bf e}_1, {\bf e}_3, \ldots, {\bf e}_{2\ell-3}\}\subseteq \{u_2,\ldots,u_s\}.$

\item If $u_1={\bf e}_{2\ell}=x_{2\ell}x_{2\ell+1}$ for some $1\leq \ell\leq m-2$,   then $2\ell+1$ belongs to $\set(\beta)$ if and only if either $u_j=x_{2\ell+1}x_{2\ell+2}$ for some $j\geq 2$ or $\{{\bf e}_{0}, {\bf e}_2, {\bf e}_4,\ldots, {\bf e}_{2\ell-2}\}\subseteq \{u_2,\ldots,u_s\}.$

\end{enumerate}
\end{Lemma}
\begin{proof}
To prove $\operatorname{set}(\beta)\subseteq [n-2]$, it suffices to show that $n-1 = 2m-1\notin \operatorname{set}(\beta)$. To this end, by Lemma~\ref{Basic}, it suffices to verify the following two claims:

(i)  $2m-2$ is not even-connected to $2m$ via any product of edges;

(ii)  $2m$ is not even-connected to itself via any product of edges.

Suppose, for contradiction, that $2m$ is even-connected to itself with respect to some product of edges. This would immediately imply the existence of a closed walk of odd length in $G$. However, since $G$ is bipartite, every closed walk in $G$ has even length, a contradiction.

Next, suppose that $2m-2$ is even-connected to $2m$ with respect to some product of edges. This connectivity would yield a closed walk of odd length starting and ending at $2m$, which again contradicts the fact that all closed walks in a bipartite graph have even length.

We now prove the assertion (1).

(1) ``$\Rightarrow$"
Suppose that \(2\ell\in\mathrm{set}(\beta)\) and \(u_j\neq x_{2\ell}x_{2\ell+1}\) for all \(2\le j\le s\). By Lemma~\ref{Basic}, there exists \(k>2\ell\) such that \(2\ell-1\) is even-connected to \(k\) with respect to \(u_2\cdots u_s\), via the walk
\[
2\ell-1 = y_0-y_1-y_2-\cdots-y_{2t-1}-y_{2t}-y_{2t+1}=k.
\]
We further assume this walk  cannot be shortened. We proceed by induction on \(\ell\) to show that \(\{\mathbf{e}_1,\mathbf{e}_3,\dots,\mathbf{e}_{2\ell-3}\}\subseteq\{u_2,\dots,u_s\}\).

For \(\ell=1\), the claim is trivial, since the set \(\{\mathbf{e}_1,\mathbf{e}_3,\dots,\mathbf{e}_{2\ell-3}\}\) is empty. Now suppose \(\ell>1\).

It is clear that \(y_1\in\{2\ell,2\ell-2\}\). If \(y_1=2\ell\), then \(y_2\) must equal \(2\ell-1\); otherwise, \(u_j=x_{y_1}x_{y_2}=x_{2\ell}x_{2\ell+1}\) for some \(j\ge2\), which would contradict the non-shortenable assumption of the walk. Thus \(y_1=2\ell-2\), whence \(y_2=2\ell-3\). It follows that \(\mathbf{e}_{2\ell-3}=x_{y_1}x_{y_2}\in\{u_2,\dots,u_s\}\). Without loss of generality, we may set \(u_2=\mathbf{e}_{2\ell-3}\).

Consider now the subwalk
\[
2\ell-3 = y_2-y_3-y_4-\cdots-y_{2t+1}=k.
\]
This subwalk certifies that \(2\ell-3\) is even-connected to \(k\) with respect to \(u_3\cdots u_s\). Note that \(2\ell-1\) does not appear in this subwalk: the original walk is non-shortenable, and \(y_0=2\ell-1\) occurs only at the starting vertex. The inductive hypothesis then yields \(\{\mathbf{e}_1,\mathbf{e}_3,\dots,\mathbf{e}_{2\ell-5}\}\subseteq\{u_3,\dots,u_s\}\). Combining this with \(u_2=\mathbf{e}_{2\ell-3}\), we conclude that \(\{\mathbf{e}_1,\mathbf{e}_3,\dots,\mathbf{e}_{2\ell-3}\}\subseteq\{u_2,\dots,u_s\}\), completing the induction.

$\Leftarrow$  If $u_j=x_{2\ell}x_{2\ell+1}$ for some $2\leq j\leq s$, then, since $2\ell+2\leq 2m$ is a vertex of $G$ such that $2\ell+2$ is adjacent to $2\ell+1$, we conclude that $2\ell\in \set(\beta)$ by Lemma~\ref{Basic}. If  $\{{\bf e}_1, {\bf e}_3, \ldots, {\bf e}_{2\ell-3}\}\subseteq \{u_2,\ldots,u_s\},$ then $2m$ is even-connected to $2\ell-1$ with respect to $u_2\cdots u_s$  by the following walk\[
2m - \underbrace{1 - 2}_{\mathbf{e}_1} - \underbrace{3 - 4}_{\mathbf{e}_3} - \cdots - \underbrace{(2\ell-3) - (2\ell-2)}_{\mathbf{e}_{2\ell-3}} - (2\ell-1).
\]
Hence, $2\ell$ belongs to $\set(\beta)$ by Lemma~\ref{Basic}.

The proof of (2) is similar to the proof of (1), we omit it.
\end{proof}

In view of this lemma, we see that the maximal set among all sets of the form \( \set(\beta) \), where \( \beta \in \mathcal{G}(I_c(G)^s) \) for some \( s \ge 1 \), is \( [n-2] \).
Next, we consider the case when $n$ is odd.

\begin{Lemma} \label{odd} Let $G$ be the odd cycle of length $n=2m+1$ and let $I$ be the complementary edge ideal of $G$. Take $\beta=\frac{\alpha^s}{u_1u_2\cdots u_s}\in I^s$, where $u_j\in I(G)$ for all $j$. Then the following conclusions hold:

\begin{enumerate}
\item If $u_1={\bf e}_{2\ell-1}=x_{2\ell-1}x_{2\ell}$ for some $1\leq \ell\leq m-1$,   then $2\ell$ belongs to $\set(\beta)$ if and only if either $u_j=x_{2\ell}x_{2\ell+1}$ for some $j\geq 2$ or $\{{\bf e}_1, {\bf e}_3, \ldots, {\bf e}_{2\ell-3}\}\subseteq \{u_2,\ldots,u_s\}.$

\item If $u_1={\bf e}_{2\ell}=x_{2\ell}x_{2\ell+1}$ for some $1\leq \ell\leq m-2$,   then $2\ell+1$ belongs to $\set(\beta)$ if and only if either $u_j=x_{2\ell+1}x_{2\ell+2}$ for some $j\geq 2$ or $\{{\bf e}_{n}, {\bf e}_2, {\bf e}_4,\ldots, {\bf e}_{2\ell-2}\}\subseteq \{u_2,\ldots,u_s\}.$

\item $2m\in \set(\beta)$ if and only if $\{{\bf e}_1,{\bf e}_3\ldots, {\bf e}_{2m-3},{\bf e}_{2m-1}\}\subseteq \{u_1,\ldots,u_s\}.$
\end{enumerate}

\begin{proof} The proofs of (1) and (2) are similar to those of Lemma~\ref{even}, so we only present the proof of (3).

(3) Suppose that $\{{\bf e}_1,{\bf e}_3\ldots, {\bf e}_{2m-3},{\bf e}_{2m-1}\}\subseteq \{u_1,\ldots,u_s\}$. We may assume that $u_1={\bf e}_{2m-1}$ without loss of generality. Then, since $2m+1$ is even-connected to $2m-1$ with respect to $u_2\cdots u_{s}$ by the following walk:
$$(2m+1)-1-2-3-4-\cdots-(2m-3)-(2m-2)-(2m-1),$$
 we conclude that $2m \in \set(\beta)$ by Lemma~\ref{Basic}.

 Conversely, if \(2m \in \text{set}(\beta)\), then there exists some \(j \in [s]\) such that \(u_j \in \{x_{2m}x_{2m+1}, x_{2m-1}x_{2m}\}\). Without loss of generality, we may assume \(j=1\). We analyze the following cases:

 (i) $u_1=x_{2m}x_{2m+1}={\bf e}_{2m}$: By  Lemma~\ref{Basic}, $2m+1$ is even-connected to itself with respect to $u_2u_3\cdots u_{s}$.  We may assume  the following  is a shortest walk by which $2m+1$ is even-connected to itself with respect to $u_2u_3\cdots u_{s}$: $$2m+1=y_0-y_1-y_2-\cdots-y_{2t-1}-y_{2t}-y_{2t+1}=2m+1, (\dag)$$

It is clear that \( y_1 \in \{1, 2m\} \). We only consider the case when \( y_1 = 1 \), as the proof for the other case is dual. Since \( y_1 = 1 \), we have \( y_2 \in \{2, 2m+1\} \). However, if \( y_2 = 2m+1 \), then the sequence \( 2m+1 - y_3 - y_4 - \cdots - (2m-1) - (2m) - (2m+1) \) also forms a walk by which \( 2m+1 \) is even-connected to itself. This contradicts the assumption that walk (\(\dag\)) is shortest. Thus, \( y_2 = 2 \). This implies \( y_3 = 3 \) or \( y_3 = 1 \). If \( y_3 = 1 \), walk (\(\dag\)) can be shortened, which is another contradiction. Therefore, \( y_3 = 3 \). By analogy, we conclude \( y_4 = 4, y_5 = 5, \ldots \). Finally, we have \( t = m \) and \( y_{2m} = 2m \), which implies \( \{\mathbf{e}_1, \mathbf{e}_3, \ldots, \mathbf{e}_{2m-1}\} \subseteq \{u_2, u_3, \ldots, u_s\} \). In particular, \( \{\mathbf{e}_1, \mathbf{e}_3, \ldots, \mathbf{e}_{2m-1}\} \subseteq \{u_1, u_2, \ldots, u_s\} \), as required.

(ii)  $u_1=x_{2m-1}x_{2m}={\bf e}_{2m-1}$: By  Lemma~\ref{Basic}, $2m+1$ is even-connected to $2m-1$ with respect to $u_2u_3\cdots u_{s}$.  We may assume  the following walk is a shortest walk by which $2m-1$ is even-connected to $2m+1$ with respect to $u_2u_3\cdots u_{s}$: $$2m-1=y_0-y_1-y_2-\cdots-y_{2t-1}-y_{2t}-y_{2t+1}=2m+1.$$
If $y_1=2m$, then $y_2=2m+1$. It follows that $2m+1-y_3-y_4-y_5-y_6-\cdots-y_{2t-1}-y_{2t}-y_{2t+1}=2m+1$ is a shortest walk that
even-connects $2m+1$ to itself with respect to $u_2\cdots u_2$. By the conclusion of (i), it follows that $\{{\bf e}_1,{\bf e}_3,\cdots {\bf e}_{2m-1}\}\subseteq \{u_2,u_3,\ldots,u_s\}$.

It remains to consider the case when $y_1=2m-2$. In this case, we can conclude, as in  case (i), that   $y_2=2m-3, y_3=2m-4,\cdots,$ and that $y_{2m-3}=2$ and  $y_{2m-2}=1$. Hence, we have $\{{\bf e}_1,{\bf e}_3,\cdots {\bf e}_{2m-1}\}\subseteq \{u_2,u_3,\ldots,u_s\}$ again.
\end{proof}
\end{Lemma}

By the preceding lemmas, we establish the following lemma, which plays a crucial role in obtaining the main result of this subsection.
\begin{Lemma} \label{EO} Let $G$ be the cycle of length $n$, whose vertices are labelled $1,2,\ldots,n$ successively. Let ${\bf e}_j$ denote the edge $x_{j}x_{j+1}$ for $j=1,\ldots,n-1$ and set ${\bf e}_0={\bf e}_n=x_nx_1$. Take $\beta=\frac{\alpha^s}{u_1u_2\cdots u_s}\in I^s$. Let $T:=\{ u_1, \ldots, u_s\}$.
We define
\[
N_1(\beta) =
\begin{cases}
\max\left\{ [2\ell] \mid \{\mathbf{e}_1, \mathbf{e}_3, \ldots, \mathbf{e}_{2\ell-1}\} \subseteq T,\ 1\leq \ell\leq m-1 \right\}, & n=2m; \\
\max\left\{ [2\ell] \mid \{\mathbf{e}_1, \mathbf{e}_3, \ldots, \mathbf{e}_{2\ell-1}\} \subseteq T,\ 1\leq \ell\leq m \right\}, & n=2m+1,
\end{cases}
\]
where the maximum is taken with respect to set inclusion. We further define
\[
N_2(\beta) =
\begin{cases}
\max\left\{ [2\ell+1] \mid \{\mathbf{e}_0, \mathbf{e}_2, \ldots, \mathbf{e}_{2\ell}\} \subseteq T,\ 1\leq \ell\leq m-2 \right\}, & n=2m; \\[2pt]
\max\left\{ [2\ell+1] \mid \{\mathbf{e}_0, \mathbf{e}_2, \ldots, \mathbf{e}_{2\ell}\} \subseteq T,\ 1\leq \ell\leq m-1 \right\}, & n=2m+1,
\end{cases}
\]
where the maximum is also taken with respect to set inclusion.

By convention, if ${\bf e}_1\notin \{ u_1, \ldots, u_s\}$ then $N_1(\beta)=\emptyset$, and if ${\bf e}_0\notin \{ u_1, \ldots, u_s\}$ then $N_2(\beta)=\emptyset$.
Then $$\set(\beta)=N_1(\beta)\cup N_2(\beta) \cup \left (\{ \min({u_j}):\ 1\leq j\leq s\}\setminus \{n-1\}\right).$$

\end{Lemma}
Before the proof, we made some remarks. If $n=2m+1$ is an odd number,  any $\beta\in \mathcal{G}(I^s)$ is expressed as $\frac{\alpha^s}{u_1u_2\cdots u_s}$ uniquely. Hence, there is no ambiguity in the definitions of $N_1(\beta)$ and $N_2(\beta)$. However, if $n$ is even, we should note that the elements of $N_1(\beta)$ and $N_2(\beta)$ depend on the way $\beta$ is written as $\frac{\alpha^s}{u_1u_2\cdots u_s}$. For example, let $\beta_1=\frac{\alpha^m}{{\bf e}_1{\bf e}_3\cdots {\bf e}_{2m-1}}$ and $\beta_2=\frac{\alpha^m}{{\bf e}_2{\bf e}_4\cdots {\bf e}_{2m}}$. Then $\beta_1=\beta_2$. However, $N_1(\beta_1)=[2m-2]$, $N_1(\beta_2)=\emptyset$, $N_2(\beta_1)=\emptyset$ and $N_2(\beta_2)=[2m-3]$.
Of course, $\set(\beta)$ is independent of the expression of $\beta$. In the above example, note that $\min({\bf e}_{2m-2})=2m-2$.  we have $\set(\beta_1)=\set(\beta_2)=[2m-2]$. Since the only real relation that occurs among ${\bf e}_1,\ldots,{\bf e}_n$ is ${\bf e}_1{\bf e}_3\cdots {\bf e}_{2m-1}= {\bf e}_2\cdots {\bf e}_{2m-2}{\bf e}_{2m}$, the set $\set(\beta)$ given in Lemma~\ref{EO} is actually independent of the expression of $\beta$.

\begin{proof} We only consider the case when \( n = 2m + 1 \) (an odd number), as the proof for the other case is analogous and even simpler. Let \( j \in \mathrm{set}(\beta) \). Then \( j \in [n-1] \) and \( j \in u_k \) for some \( 1 \leq k \leq s \). If \( j = n-1 = 2m \), then by Lemma~\ref{odd}.(3), \( \{{\bf e}_1, {\bf e}_3, \ldots, {\bf e}_{2m-3}, {\bf e}_{2m-1}\} \subseteq \{u_1, \ldots, u_s\} \), which implies \( j \in N_1(\beta) \). Thus, we may assume \( j \leq 2m - 1 \). If \( j = \min(u_k) \) for some \( k\in [s] \), then \( j \in \{\min( u_j) :\ 1 \leq j \leq s\} \setminus \{n-1\} \), and we are done.

 Suppose now that $j\leq 2m- 1$ and  \( j \neq \min(u_k) \) for any \( k = 1, \ldots, s \). By Lemma~\ref{odd}, we conclude that if \( j = 2\ell \) for some \( \ell \), then \( \{{\bf e}_1, {\bf e}_3, \ldots, {\bf e}_{2\ell-1}\} \subseteq \{u_1, \ldots, u_s\} \);
and if \( j = 2\ell + 1 \) for some \( \ell \), then \( \{{\bf e}_n, {\bf e}_2, \ldots, {\bf e}_{2\ell}\} \subseteq \{u_1, \ldots, u_s\} \).
Furthermore, since \( j \leq 2m-1 \), we have:
If \( j = 2\ell + 1 \), then \( \ell \leq m - 1 \);
If \( j = 2\ell \), then \( \ell \leq m - 1 \).
This implies \( j \in N_1(\beta) \cup N_2(\beta) \), so the inclusion \( \subseteq \) holds. The converse inclusion follows directly from Lemma~\ref{odd}.
\end{proof}

As an immediate  consequence of Lemma~\ref{EO}, we obtain:

\begin{Corollary} Keep the notation as before.   Let $\beta=\frac{\alpha^s}{{\bf e}_1^{s_1}{\bf e}_2^{s_2}\cdots {\bf e}_n^{s_n}}\in \mathcal{G}(I^s)$. Define $t_j=1$ if $s_j>0$ and   $t_j=0$ if $s_j=0$, for $1\leq j\leq n$.
Then $\set(\beta)=\set (\gamma)$, where $\gamma=\frac{\alpha^t}{{\bf e}_1^{t_1}{\bf e}_2^{t_2}\cdots {\bf e}_n^{t_n}}$ and $t=t_1+\cdots+t_n$.
\end{Corollary}
By this corollary, it is straightforward to verify that  $\mathrm{HS}_i(\mathcal{R(}I))$ is generated in degree $\leq n$ for all $i\geq 1$. This result is sharpened in Corollary~\ref{degree}.

Hereafter, we focus on computing the homological shift ideals of powers of the complementary edge ideal associated with a cycle graph. Note that if $s<\frac{i}{2}$, then $\mathrm{HS}_i(I^s)=0$, so we only consider the case when $s\geq \frac{i}{2}$. In the rest of this section, we always use the convention that \( I^{t} = 0 \) for all \( t < 0 \).

\begin{Lemma} \label{half}
Let \( G \) be a cycle of length \( n \) and \( I \) be its complementary edge ideal. For \( 1 \leq i \leq n-1 \) and \( s \geq \frac{i}{2} \), the following inclusion holds:
\begin{equation} \label{subset}
\mathrm{HS}_i(I^s) \subseteq \sum_{0 \leq k \leq \frac{i}{2}} \left( \sum_{\substack{F \subseteq [n] \\ |F| = i - 2k}} \frac{\alpha^{i - k}}{{\bf x}_F} \cdot I^{s - i + k} \right).
\end{equation}
\end{Lemma}

\begin{proof}
We only prove the case where \( n = 2m \) (even cycle), as the argument for \( n = 2m + 1 \) (odd cycle) is analogous. Fix \( i \in [2m - 2] \) and \( s \geq \frac{i}{2} \). Let \( \beta \in \mathcal{G}(I^s) \) with \( i \leq |\mathrm{set}(\beta)| \), and take a subset \( Q \subseteq \mathrm{set}(\beta) \) with \( |Q| = i \). It suffices to show that \( {\bf x}_Q \beta \) lies in the right-hand side of \eqref{subset}.

Write $\beta=\frac{\alpha^s}{u_1\cdots u_s}$ with $u_j\in \mathcal{G}(I(G))$ for all $j=1,\ldots, s$.   Set \( T := \{u_1, \ldots, u_s\} \). Denote $$\ell_1:=\max\left \{1\leq \ell\leq m-1:\ \{{\bf e}_1,{\bf e}_3,\ldots {\bf e}_{2\ell-1}\}\subseteq T \right\}$$ and  $$\ell_2:=\max\{1\leq \ell\leq m-2:\ \{{\bf e}_0,{\bf e}_2,\ldots {\bf e}_{2\ell}\}\subseteq T\}.$$
By convention, $\max \emptyset=-\infty$. By Lemma~\ref{EO}, $N_1(\beta)\cup N_2(\beta)=[2\ell_1]\cup [2\ell_2+1]$.  We consider two cases:

(1) \( Q\cap \left( [2\ell_1] \cup [2\ell_2 + 1] \right) = \emptyset \): By Lemma~\ref{EO}, there exist distinct edges \( v_1, \ldots, v_i \in T \) such that \( Q = \{\min(v_1), \ldots, \min(v_i)\} \). We may write \( u_1u_2\cdots u_s = (v_1\cdots v_i)(u_1\cdots u_{s-i}) \) (after reindexing if necessary). Note that \( \max(v_1), \ldots, \max(v_i) \) are pairwise distinct; let \( F \) denote the \( i \)-element set \( \{\max(v_1), \ldots, \max(v_i)\} \).

Then
\[
{\bf x}_Q \beta = \frac{\alpha^s}{{\bf x}_F \cdot u_1\cdots u_{s-i}} = \frac{\alpha^i}{{\bf x}_F} \cdot \frac{\alpha^{s-i}}{u_1\cdots u_{s-i}}.
\]
Hence, \( {\bf x}_Q \beta \) belongs to the \( k = 0 \) component of the right-hand side of \eqref{subset}.

(2) $Q_1:=Q \cap \left( [2\ell_1] \cup [2\ell_2 + 1] \right) \neq \emptyset$: In this case, we can decompose \( Q \) as
\[
Q = Q_1 \sqcup \{j_{p+1} < \cdots < j_i\},
\]
where \( |Q_1| = p \), \( j_{p+1} > \max Q_1 \), and for all \( k = p+1, \ldots, i \),
\[
{\bf e}_{j_k} \in \{u_1, u_2, \ldots, u_s\} \setminus \{{\bf e}_1, {\bf e}_3, \ldots, {\bf e}_{2\ell_1-1}, {\bf e}_0, {\bf e}_2, \ldots, {\bf e}_{2\ell_2}\}.
\]
Note that if $p=i$, then $\{j_{p+1} < \cdots < j_i\}$ is an empty set.

First, consider the subcase \( \ell_2 \geq \ell_1 \). Here, \( Q_1 \subseteq {\bf e}_0 \cup {\bf e}_2 \cup \cdots \cup {\bf e}_{2\ell_2} \). Let \( \Omega := \{{\bf e}_0, \ldots, {\bf e}_{2\ell_2}\} \). There exist edges \( v_1, \ldots, v_q \in \Omega \setminus \{{\bf e}_0\} \) and edges \( f_1, \ldots, f_r \in \Omega \) such that
\[
{\bf x}_{Q_1} = v_1v_2\cdots v_q \cdot x_{j_1}\cdots x_{j_r},
\]
where the edges \( v_1, \ldots, v_q, f_1, \ldots, f_r \) are pairwise disjoint, and \( j_t \in f_t \) for \( t = 1, \ldots, r \). Note that \( 2q + r = p \) and \( q + r \leq \ell_2 \). For each \( t = 1, \ldots, r \), let \( i_t \in [2\ell_2 + 1] \) be such that \( f_t = x_{j_t}x_{i_t} \).

Since \( \{v_1, \ldots, v_q, f_1, \ldots, f_r, {\bf e}_{j_{p+1}}, \ldots, {\bf e}_{j_i}\} \subseteq \{u_1, \ldots, u_s\} \), we may assume, up to reindexing the elements \( u_1,\dots,u_s \), that
\[
u_1\cdots u_s = \left( v_1\cdots v_q f_1\cdots f_r {\bf e}_{j_{p+1}}\cdots {\bf e}_{j_i} \right) u_1\cdots u_{s - i + q}.
\]
It then follows that
\[
{\bf x}_Q \beta = \frac{\alpha^s}{u_1\cdots u_{s - i + q} \cdot (x_{i_1}\cdots x_{i_r}) \cdot (x_{j_{p+1}+1} \cdots x_{j_i + 1})}.
\]
Let us set \( F := \{i_1, \ldots, i_r, j_{p+1}+1, \ldots, j_i + 1\} \). Then \( F \) is an \( (i - q) \)-subset of \( [n] \). With this notation, we have $$ {\bf x}_Q \beta=\frac{\alpha^{i-q}}{(x_{i_1}\cdots x_{i_r}) \cdot (x_{j_{p+1}+1} \cdots x_{j_i + 1})} \cdot \frac{\alpha^{s-i+q}}{u_1\cdots u_{s - i + q}} \in \frac{\alpha^{i - q}}{{\bf x}_F} I^{s - i + q},$$ which implies that \( {\bf x}_Q \beta \) lies in the component corresponding to \( k = q \) on the right-hand side of \eqref{subset}.

The proof for \( \ell_1 > \ell_2 \) is analogous, so the argument is complete.
\end{proof}

In general, the converse of the inclusion \eqref{subset} fails to hold. For the converse inclusion to hold, it is necessary to remove certain components from the right-hand side of \eqref{subset}. This is accomplished by proving that any subset  $F \subseteq \set{\left(\frac{\alpha^s}{u_1\ldots u_s}\right)}$ with $|F| = \left\lceil \frac{n}{2} \right\rceil + k - 1$ contains at least $k$ pairwise disjoint edges among $u_1, \ldots, u_s$.

Note that for any integer \( k \), \( k \leq \left\lfloor \frac{i}{2} \right\rfloor \) if and only if \( k \leq \frac{i}{2} \), and \( k \geq \left\lceil \frac{i}{2} \right\rceil \) if and only if \( k \geq \frac{i}{2} \).

\begin{Proposition}\label{half1}
Let $ G $ be a cycle of length $ n $ and $ I $ be its complementary edge ideal. For all integers $ 1 \leq i < n $ and $ s \geq \left\lfloor \frac{i}{2} \right\rfloor $, the following inclusion holds:
\begin{equation}\label{refine}
\mathrm{HS}_i(I^s) \subseteq
\sum_{\substack{\max\left\{i - \left\lceil \frac{n}{2} \right\rceil +1, 0\right\} \leq k \leq \left\lfloor \frac{i}{2} \right\rfloor}}
\left(
  \sum_{\substack{F \subseteq [n] \\ |F| = i - 2k}}
  \frac{\alpha^{i - k}}{{\bf x}_F} \cdot I^{s - i + k}
\right).
\end{equation}
\end{Proposition}

\begin{proof}
Set \(q:=\max\left\{i-\lceil\frac{n}{2}\rceil+1,0\right\}\). If \(i\le\lceil\frac{n}{2}\rceil-1\), then \(q=0\), so the right-hand side of \eqref{refine} coincides with the left-hand side of \eqref{subset}, and the result follows immediately from Lemma~\ref{half}. For \(i>\lceil\frac{n}{2}\rceil-1\), we focus on the case \(n=2m\) since the argument for \(n=2m+1\) is analogous. Since \(q=i-m+1\) in this case, we have \(i=m-1+q\). Take \(\beta\in\mathcal{G}(I^s)\) with \(i\le|\mathrm{set}(\beta)|\), and let \(Q\subseteq\mathrm{set}(\beta)\) be a subset of size \(i\); it suffices to show that \(\mathbf{x}_Q\beta\) lies in the right-hand side of \eqref{refine}.

To begin with, we write \( \beta = \frac{\alpha^s}{u_1 \cdots u_s} \), where \( u_j \in \mathcal{G}(I(G)) \) for all \( j = 1, \ldots, s \), and set \( T := \{u_1, \ldots, u_s\} \).

 Denote
$$\ell_1:=\max\left \{1\leq \ell\leq m-1:\ \{{\bf e}_1,{\bf e}_3,\ldots, {\bf e}_{2\ell-1}\}\subseteq T\right\}$$
and
$$\ell_2:=\max\left\{1\leq \ell\leq m-2:\ \{{\bf e}_0,{\bf e}_2,\ldots, {\bf e}_{2\ell}\}\subseteq T\right\}.$$
By convention, $\max \emptyset=-\infty$. Suppose first that $\ell_1>\ell_2$. By Lemma~\ref{EO}, it follows that
$$\mathrm{set}(\beta)=[2\ell_1]\cup \left(\{\min (u_j):\ j=1,\ldots,s\}\setminus \{n-1\}\right).$$

Note that the edges ${\bf e}_1,{\bf e}_3,\ldots,{\bf e}_{2m-3}$ are pairwise disjoint, and their union is the set $[2m-2]$. Since $Q\subseteq \mathrm{set}(\beta)\subseteq [2m-2]$ and $|Q|=m-1+q$, by the pigeonhole principle, there are  $k\geq q$ edges among ${\bf e}_1,{\bf e}_3,\ldots,{\bf e}_{2m-3}$ that are entirely contained in $Q$. Let these $k$ edges be ${\bf e}_{2j_1-1},{\bf e}_{2j_2-1},\ldots,{\bf e}_{2j_k-1}$ with $j_1<j_2<\cdots<j_k$.

There are three possible scenarios for the indices $j_t$ relative to $\ell_1$:
1. $\ell_1<j_1$,
2. $j_k\leq \ell_1$,
3. There exists some $1\leq r\leq k-1$ such that $j_r\leq \ell_1<j_{r+1}$ (indices straddle $\ell_1$).

We only prove the third scenario, as the first two follow by analogous reasoning. For $r+1\leq t\leq k$, since $j_t> \ell_1$, we have $2j_t-1>2\ell_1$. Given that $2j_t-1\in Q\subseteq \mathrm{set}(\beta)$, Lemma~\ref{EO} implies ${\bf e}_{2j_t-1}\in T$. For $1\leq t\leq r$, since $j_t\leq \ell_1$, the edge ${\bf e}_{2j_t-1}$ is in $\{{\bf e}_1,{\bf e}_3,\ldots,{\bf e}_{2\ell_1-1}\}$, which is contained in $T$ by the definition of $\ell_1$. Thus, ${\bf e}_{2j_1-1},\ldots,{\bf e}_{2j_k-1}$ are pairwise disjoint edges in $T$ whose union is contained in $Q$. Let \( Q_1 = Q \setminus ({\bf e}_{2j_1-1} \sqcup \cdots \sqcup {\bf e}_{2j_k-1}) \). Then \( |Q_1| = i - 2k \); write \( Q_1 = \{\theta_1, \ldots, \theta_p\} \) where \( p = i - 2k \). For each \( \theta_j \in Q_1 \):

- If \( \theta_j \leq 2\ell_1 \), choose \( v_j \in \{{\bf e}_1, {\bf e}_3, \ldots, {\bf e}_{2\ell_1-1}\} \) such that \( \theta_j \in v_j \);

- If \( \theta_j > 2\ell_1 \), set \( v_j = \{\theta_j, \theta_j + 1\} \).

It is straightforward to verify that each \( v_j \in T \). Let \( F = \bigsqcup_{j=1}^p \left( v_j \setminus \{\theta_j\} \right) \). After reindexing, we may write
\[
u_1 \cdots u_s = \left( {\bf e}_{2j_1-1} \cdots {\bf e}_{2j_k-1} \right) \left( v_1 \cdots v_p \right) \left( u_1 \cdots u_{s - i + k} \right).
\]
Then
\[
{\bf x}_Q \beta = \frac{\alpha^s}{(u_1 \cdots u_{s - i + k}) \cdot {\bf x}_F} \in \frac{\alpha^{i - k}}{{\bf x}_F} \cdot I^{s - i + k}.
\]
Hence, \( {\bf x}_Q \beta \) belongs to the right-hand side of \eqref{refine}.

Next, we consider the case when \( \ell_1 \leq \ell_2 \), which is much more complicated than the first case. By Lemma~\ref{EO}, we have
\[
\mathrm{set}(\beta) = [2\ell_2 + 1] \cup \left( \{\min (u_j) :\ j = 1, \ldots, s\} \setminus \{n - 1\} \right).
\]
Consider the \( m - 2 \) edges \( {\bf e}_2, {\bf e}_4, \ldots, {\bf e}_{2m - 4} \). Note that \( \mathrm{set}(\beta) \subseteq ({\bf e}_2 \sqcup {\bf e}_4 \sqcup \cdots \sqcup {\bf e}_{2m - 4}) \cup \{1, 2m - 2\} \), and since \( Q \subseteq \mathrm{set}(\beta) \), it follows that
\[
Q \subseteq ({\bf e}_2 \sqcup {\bf e}_4 \sqcup \cdots \sqcup {\bf e}_{2m - 4}) \cup \{1, 2m - 2\}.
\]
This inclusion implies
\[
|Q \cap ({\bf e}_2 \sqcup {\bf e}_4 \sqcup \cdots \sqcup {\bf e}_{2m - 4})| \geq |Q| - 2.
\]
Since \( |Q| = m - 1 + q \), substituting this in gives
\[
|Q \cap ({\bf e}_2 \sqcup {\bf e}_4 \sqcup \cdots \sqcup {\bf e}_{2m - 4})| \geq (m - 1 + q) - 2 = m - 3 + q.
\]
Equality holds if and only if \( Q \cap \{1, 2m - 2\} = \{1, 2m - 2\} \), i.e., \( \{1, 2m - 2\} \subseteq Q \).

If \( |Q \cap ({\bf e}_2 \sqcup {\bf e}_4 \sqcup \cdots \sqcup {\bf e}_{2m - 4})| \geq m - 2 + q \), then by the pigeonhole principle, there are at least \( q \) edges among \( {\bf e}_2, {\bf e}_4, \ldots, {\bf e}_{2m - 4} \) that are entirely contained in \( Q \). The proof then proceeds analogously to the case where \( \ell_1 > \ell_2 \).

Assume now that \( |Q \cap ({\bf e}_2 \sqcup {\bf e}_4 \sqcup \cdots \sqcup {\bf e}_{2m-4})| = m - 3 + q \), which implies \( \{1, 2m - 2\} \subseteq Q \). By the pigeonhole principle, there are at least \( q - 1 \) edges among \( {\bf e}_2, {\bf e}_4, \ldots, {\bf e}_{2m-4} \) that are entirely contained in \( Q \). If there are more than \( q - 1 \) such edges, we are done (as we would have at least \( q \) edges). Thus, we may assume there are exactly \( q - 1 \) such edges, denoted \( {\bf e}_{2j_1}, \ldots, {\bf e}_{2j_{q-1}} \) where \( j_1 < j_2 < \cdots < j_{q-1} \). For the remaining edges \( {\bf e}_{2t} \) with \( t \in [m-2] \setminus \{j_1, \ldots, j_{q-1}\} \), it follows that \( |Q \cap {\bf e}_{2t}| = 1 \).

The indices \( j_1, \ldots, j_{q-1} \) satisfy one of three scenarios relative to \( \ell_2 \):
1. \( \ell_2 < j_1 \) (all indices are greater than \( \ell_2 \)),
2. \( j_{q-1} \leq \ell_2 \) (all indices are less than or equal to \( \ell_2 \)),
3. There exists some \( 1 \leq r \leq q-2 \) such that \( j_r \leq \ell_2 < j_{r+1} \) (indices straddle \( \ell_2 \)).

We focus on the third scenario (the others follow similarly, with minor adjustments). First, we claim \( 2\ell_2 + 2 \notin Q \): if \( 2\ell_2 + 2 \in Q \), then \( 2\ell_2 + 2 \in \mathrm{set}(\beta) \), so Lemma~\ref{EO} implies \( {\bf e}_{2\ell_2 + 2} \in T \), contradicting the maximality of \( \ell_2 \).

Set \( p = q - 1 - r \). Since \( 2m - 2 \in Q \), we calculate the size of \( Q \cap [2\ell_2 + 3, 2m - 2] \):
\[
|Q \cap [2\ell_2 + 3, 2m - 2]| = (m - 1 + q) - |Q \cap [1, 2\ell_2 + 2]|.
\]
By construction, \( |Q \cap [1, 2\ell_2 + 2]| = |Q \cap [1, 2\ell_2 + 1]| = \ell_2 + r + 1 \) (including \( 1 \), plus \( 2r \) elements from the \( r \) edges \( {\bf e}_{2j_1}, \ldots, {\bf e}_{2j_r} \subseteq [2\ell_2 + 2] \), and \( \ell_2 - r \) elements from \( \ell_2 - r \) edges \( {\bf e}_{2t} \) with \( 1 \leq t \leq \ell_2 \) and \( t \notin \{j_1, \ldots, j_r\} \)). Substituting this gives:
\[
|Q \cap [2\ell_2 + 3, 2m - 2]| = (m - 1 + q) - (\ell_2 + r + 1) = m - \ell_2 - 1 + p,
\]
(using \( p = q - 1 - r \) and simplifying).

Next, consider the \( m - 2 - \ell_2 \) edges \( {\bf e}_{2\ell_2 + 3}, {\bf e}_{2\ell_2 + 5}, \ldots, {\bf e}_{2m - 3} \), whose disjoint union is exactly \( [2\ell_2 + 3, 2m - 2] \). Since \( |Q \cap [2\ell_2 + 3, 2m - 2]| = (m - 2 - \ell_2) + (p + 1) \), the pigeonhole principle guarantees at least \( p + 1 \) of these edges are fully contained in \( Q \). These \( p + 1 \) edges, together with the \( r \) edges \( {\bf e}_{2j_1}, \ldots, {\bf e}_{2j_r} \), form \( (p + 1) + r = q \) pairwise disjoint edges in \( T \) whose union is contained in \( Q \). We denote such $q$ edges by $f_1,\ldots,f_q$.  Let \( Q_1 \) be the remainder of \( Q \) after removing all vertices contained in these \( q \) edges.

Write \( Q_1 = \{\theta_1, \ldots, \theta_p\} \) where \( p = i - 2q\). For each \( \theta_j \in Q_1 \):

- If \( \theta_j \leq 2\ell_1 \), choose \( v_j \in \{{\bf e}_0, {\bf e}_2, \ldots, {\bf e}_{2\ell_2}\} \) such that \( \theta_j \in v_j \);

- If \( \theta_j > 2\ell_2+1 \), set \( v_j = \{\theta_j, \theta_j + 1\} \).

It is straightforward to verify that each \( v_j \in T \). Let \( F = \bigsqcup_{j=1}^p \left( v_j \setminus \{\theta_j\} \right) \). Then $|F|=p=i-2q$.  After reindexing, we may write
\[
u_1 \cdots u_s = \left( f_1\cdots f_q \right) \left( v_1 \cdots v_p \right) \left( u_1 \cdots u_{s - i + q} \right).
\]
Then
\[
{\bf x}_Q \beta = \frac{\alpha^s}{(u_1 \cdots u_{s - i + q}) \cdot {\bf x}_F} \in \frac{\alpha^{i - q}}{{\bf x}_F} \cdot I^{s - i + q}.
\]
Hence, \( {\bf x}_Q \beta \) belongs to $k=q$ component of the right-hand side of \eqref{refine}. This completes the proof.
\end{proof}

To obtain  the converse inclusion, we need the following lemma.

\begin{Lemma} \label{lasttwo}
Let \( G \) be a cycle of length \( n \) and \( I \) be its complementary edge ideal. Suppose \( n = 2m \) or \( n = 2m+1 \). Given a subset \( F \subseteq [n] \) with \( |F| = p \) and an integer \( k \geq 0 \) satisfying \( k + p < \frac{n}{2} \), set \( i := 2k + p \). Then there exist pairwise distinct edges \( u_1, \ldots, u_{k+p} \) and a subset \( E \subseteq \mathrm{set}\left( \frac{\alpha^{k+p}}{u_1\cdots u_{k+p}} \right) \) with \( |E| = i \) such that
\[
\frac{\alpha^{k+p}}{\mathbf{x}_F} = \mathbf{x}_E \cdot \frac{\alpha^{k+p}}{u_1\cdots u_{k+p}} \in \mathrm{HS}_i(I^{k+p}).
\]
\end{Lemma}

\begin{proof}
We only consider the case \( n = 2m \); the case \( n = 2m + 1 \) is analogous. Write \( F= \{j_1 < j_2 < \cdots < j_p\} \). We analyze two cases:

(1) \( n \notin F \): Define \[ t := \max\left\{a : j_1 \leq 2k + 2,\ j_2 \leq 2k + 4,\ \ldots,\ j_a \leq 2k + 2a\right\}. \]
By convention, set \( t := 0 \) when \( j_1 > 2k + 2 \). Then \( 1 \leq j_1 < \cdots < j_t \leq 2k + 2t \leq 2m - 2 \) and \( j_{p} > \cdots > j_{t+1} > 2k + 2t + 2 \). Set
\[
\beta := \frac{\alpha^{k+p}}{(\mathbf{e}_1\mathbf{e}_3 \cdots \mathbf{e}_{2k+2t-1})(\mathbf{e}_{j_{t+1}-1} \cdots \mathbf{e}_{j_p-1})},
\]
and \( E := ([2k + 2t] \setminus \{j_1, \ldots, j_t\}) \sqcup \{j_{t+1} - 1, \ldots, j_p - 1\}. \)
By Lemma~\ref{EO}, \( E \subseteq \mathrm{set}(\beta) \). Moreover, \( |E| = 2k + p = i \), and \( \mathbf{x}_E \cdot \beta = \frac{\alpha^{k+p}}{\mathbf{x}_F} \in \mathrm{HS}_i(I^{k+p}) \), as required.

(2) \( n \in F \): Define \( t := \max\left\{a : j_1 \leq 2k + 3,\ j_2 \leq 2k + 5,\ \ldots,\ j_a \leq 2k + 2a + 1\right\},\)
and set \( t := 0 \) when \( j_1 > 2k + 3 \). Then \( 0 \leq t \leq p - 1 \), \( \{j_1, \ldots, j_t\} \subseteq [2k + 2t + 1] \subseteq [2m - 1] \), and \( j_{t+1} > 2k + 2t + 3 \). Set
\[
\beta := \frac{\alpha^{k+p}}{(\mathbf{e}_0\mathbf{e}_2 \cdots \mathbf{e}_{2k+2t})(\mathbf{e}_{j_{t+1}-1} \cdots \mathbf{e}_{j_{p-1}-1})}.
\]
and \( E := \left( [2k + 2t + 1] \setminus \{j_1, \ldots, j_t\} \right) \sqcup \{j_{t+1} - 1, \ldots, j_{p-1} - 1\}. \)
Then \( |E| = 2k + p = i \), and by Lemma~\ref{EO}, \( E \subseteq \mathrm{set}(\beta) \). It follows that \( \frac{\alpha^{k+p}}{\mathbf{x}_F} = \mathbf{x}_E \cdot \beta \in \mathrm{HS}_i(I^{k+p}) \), as required.
\end{proof}

We now are ready to present the main result of this subsection.

\begin{Theorem}\label{express1}
Let $G$ be a cycle of length $n$ and $I$ be its complementary edge ideal. Then, for all integers $1 \leq i < n$ and $s \geq \left\lfloor \frac{i}{2} \right\rfloor$, the following equality holds:
\begin{equation}\label{main1}
\mathrm{HS}_i(I^s)=\sum_{\substack{\max\left\{i - \left\lceil \frac{n}{2} \right\rceil + 1, 0\right\} \leq k \leq \left\lfloor \frac{i}{2} \right\rfloor}}
\left(
  \sum_{\substack{F \subseteq [n] \\ |F| = i - 2k}}
  \frac{\alpha^{i - k}}{\mathbf{x}_F} \cdot I^{s - i + k}
\right).
\end{equation}
\end{Theorem}

\begin{proof}
For any $k$ satisfying
\[
\max\left\{i - \left\lceil \frac{n}{2} \right\rceil + 1, 0\right\} \leq k \leq \left\lfloor \frac{i}{2} \right\rfloor,
\]
let $p = i - 2k$. Then $k + p = i - k$, and it follows that $k + p \leq \left\lceil \frac{n}{2} \right\rceil - 1$. By Lemma~\ref{lasttwo}, for any $F \subseteq [n]$ with $|F| = p$, we have $\frac{\alpha^{i - k}}{\mathbf{x}_F} \in \mathrm{HS}_i(I^{i - k})$. Applying \cite[Proposition 1.3]{FQ1}, we deduce that
\[
\frac{\alpha^{i - k}}{\mathbf{x}_F} \cdot I^{s - i + k} \subseteq  \mathrm{HS}_i(I^s).
\]
Thus, every term in the right-hand side is contained in $\mathrm{HS}_i(I^s)$, whence the inclusion $\supseteq$ holds. The reverse inclusion follows immediately from Proposition~\ref{half}. This completes the proof.
\end{proof}
We remark that Equality~(\ref{main1}) still holds even when \(s < \left\lfloor \frac{i}{2} \right\rfloor\). In this case, \( \mathrm{HS}_i(I^s) = 0 \) by Theorem~\ref{Cycle}, and the right-hand side of this equality is also zero, since \( s - i + k < 0 \) for all \( k < \left\lfloor \frac{i}{2} \right\rfloor \).

Theorem~\ref{express1}  yields the generating degree of $\mathrm{HS}_i(\mathcal{R}(I))$.
\begin{Corollary}\label{degree}
Let $ G $ be a cycle of length $n$ and $ I $ be its complementary edge ideal. For all integers $ 1 \leq i < \frac{n}{2} $, the $i$-th homological shift algebra $\mathrm{HS}_i(\mathcal{R}(I))$ is generated in degree $\leq i-q$, where $q=\max\left\{i - \left\lceil \frac{n}{2} \right\rceil + 1, 0\right\}$.
\end{Corollary}

\begin{proof}
It is straightforward to verify that $\mathrm{HS}_i(I^s)=I^{s-i+q}\cdot  \mathrm{HS}_i(I^{i-q})$ for any $s\geq i-q$. The desired conclusion then follows directly from this identity.
\end{proof}

\begin{Remark}\em
Let \( G \) be a cycle of length \( n \). For each \( j = 1, \ldots, n \), let \( L_j \) denote the path graph with vertex set \( [n] \), obtained by deleting the edge \( \{j-1, j\} \) from \( G \). Here, we adopt the convention that vertex \( 0 \) corresponds to vertex \( n \) (so the edge deleted for \( L_1 \) is \( \{n, 1\} \)). Note that \( L_1, \ldots, L_n \) are exactly all the spanning trees of \( G \). For each \( j \), let \( I_j \) denote the complementary edge ideal of \( L_j \). By Lemma~\ref{HSsubset}, we immediately obtain the inclusion
\[
\sum_{j=1}^n \mathrm{HS}_i(I_j^s) \subseteq \mathrm{HS}_i(I^s).
\]
A natural question then arises: {\bf Which component of the decomposition of \( \mathrm{HS}_i(I^s) \) given in Formula (\ref{main1}) does the sum \( \sum_{j=1}^n \mathrm{HS}_i(I_j^s) \) correspond to?}

Suppose \( 1 \leq i < \frac{n}{2} \). For this range of \( i \), \( \max\left\{i - \left\lceil \frac{n}{2} \right\rceil + 1, 0\right\} = 0 \), since \( i < \frac{n}{2} \leq \left\lceil \frac{n}{2} \right\rceil \) implies \( i - \left\lceil \frac{n}{2} \right\rceil + 1 < 1 \). Theorem~\ref{express1} thus simplifies to
\[
\mathrm{HS}_i(I^s) = \sum_{0 \leq k \leq \left\lfloor \frac{i}{2} \right\rfloor} \left( \sum_{\substack{F \subseteq [n] \\ |F| = i - 2k}} \frac{\alpha^{i - k}}{\mathbf{x}_F} \cdot I^{s - i + k} \right).
\]
We claim that \( \sum_{j=1}^n \mathrm{HS}_i(I_j^s) \) is contained in the \( k=0 \) component of the above decomposition. In other words, the following inclusion holds for all \( s \geq i \):
\[
\sum_{j=1}^n \mathrm{HS}_i(I_j^s) \subseteq \sum_{\substack{F \subseteq [n] \\ |F| = i}} \frac{\alpha^i}{\mathbf{x}_F} \cdot I^{s - i}.
\]
Moreover, this inclusion becomes an equality when \( s = i \).

First, we verify the inclusion \( \subseteq \). By Theorem~\ref{3.1}, \( \mathrm{HS}_i(I_j^s) \) has the following decomposition:
\[
\mathrm{HS}_i(I_1^s) = \sum_{\substack{F \subseteq [n] \setminus \{1, n\} \\ |F| = i}} \frac{\alpha^i}{\mathbf{x}_F} I_1^{s-i},
\]
and for \( j = 2, \ldots, n \),
\[
\mathrm{HS}_i(I_j^s) = \sum_{\substack{F \subseteq [n] \setminus \{j-1, j\} \\ |F| = i}} \frac{\alpha^i}{\mathbf{x}_F} I_j^{s-i}.
\]
Since \( I_j \subseteq I \) for each \( j \), we have \( I_j^{s-i} \subseteq I^{s-i} \). This implies that every term in \( \sum_{j=1}^n \mathrm{HS}_i(I_j^s) \) is contained in the \( k=0 \) component of \( \mathrm{HS}_i(I^s) \), proving the inclusion \( \subseteq \).

Next, we establish the reverse inclusion \( \supseteq \) in the case \( s = i \). Recall that a \textit{vertex cover} of a graph \( G \) is a subset \( C \subseteq V(G) \) such that every edge of \( G \) is incident to at least one vertex in \( C \). Since \( i < \frac{n}{2} \) and the minimum vertex cover of a cycle \( G \) has size \( \left\lceil \frac{n}{2} \right\rceil \), \( F \) cannot be a vertex cover of \( G \). By definition of vertex covers, there exists at least one edge \( \{j_0-1, j_0\} \) of \( G \) such that \( F \cap \{j_0-1, j_0\} = \emptyset \). It follows that \( F \subseteq [n] \setminus \{j_0-1, j_0\} \), so
\[
\left(\frac{\alpha^i}{\mathbf{x}_F}\right) \subset \mathrm{HS}_i(I_{j_0}^i).
\]
Thus, every term in the \( k=0 \) component (when \( s=i \)) is contained in \( \sum_{j=1}^n \mathrm{HS}_i(I_j^i) \), proving the inclusion \( \supseteq \). Combining both inclusions, the claim is established.
\end{Remark}

For the largest \( i \) satisfying \( \mathrm{HS}_i(\mathcal{R}(I)) \neq 0 \), \( \mathrm{HS}_i(I^s) \) admits a simple expression.

\begin{Corollary}
Let \( G \) be a cycle of length \( n \) and \( I \) denote its complementary edge ideal.

{\em (1)} If \( n = 2m + 1 \), then \( \mathrm{HS}_{2m}(I^s) = \alpha^m \cdot I^{s - m} \) for all integers \( s \geq m \). In particular, \( \mathrm{HS}_{2m}(\mathcal{R}(I)) \) is generated in degree \( m \).

{\em (2)} If \( n = 2m \), then \( \mathrm{HS}_{2m - 2}(I^s) = \alpha^{m - 1} \cdot I^{s - m + 1} \) for all integers \( s \geq m - 1 \). In particular, \( \mathrm{HS}_{2m - 2}(\mathcal{R}(I)) \) is generated in degree \( m - 1 \).
\end{Corollary}

\begin{proof}
(1) In Formula (\ref{main1}) of Theorem~\ref{express1}, the range of \( k \) is given by
\[
\max\left\{ i - \left\lceil \frac{n}{2} \right\rceil + 1,\ 0 \right\} \leq k \leq \left\lfloor \frac{i}{2} \right\rfloor.
\]
For the case where \( n = 2m + 1 \) and \( i = 2m \), we have \( \left\lceil \frac{n}{2} \right\rceil = m + 1 \). The left-hand bound of the above inequality simplifies to
\[
\max\{2m - (m + 1) + 1,\ 0\} = \max\{m,\ 0\} = m,
\]
while the right-hand bound is \( \left\lfloor \frac{2m}{2} \right\rfloor = m \). Therefore, \( k = m \) is the only feasible value of \( k \), so \( i - 2k = 0 \). The equality \( \mathrm{HS}_{2m}(I^s) = \alpha^m \cdot I^{s - m} \) thus follows immediately from Formula (\ref{main1}), and it then follows that \( \mathrm{HS}_{2m}(\mathcal{R}(I)) \) is generated precisely in degree \( m \).

The proof of part (2) is analogous to that of part (1).
\end{proof}

\medskip
\hspace{-6mm} {\bf Acknowledgments}

	\vspace{3mm}
	\hspace{-6mm}
	The first author is  supported by the National Natural Science Foundation of China (No. 11971338).
	The third author is  supported by the Natural Science Foundation of Jiangsu Province (No. BK20221353) and the National Natural Science Foundation of China (No. 12471246). The authors wish to express their gratitude to the computer algebra system  \cite{F2}, which provided a wealth of examples that facilitated the development of our research ideas and the verification of our results.


\begin{thebibliography}{99}




		\bibitem{AB} A. Banerjee, {\it The regularity of powers of edge ideals}, J. Algebraic Combin., {\bf 41} (2015) 303--321.
		
	\bibitem{Ba1}	 S. Bayati, {\it Multigraded shifts of matroidal ideals}, Archiv der Mathematik,  {\bf  111} (2018), 239--246.



\bibitem{Ba2} S. Bayati, {\it A quasi-additive Property of Homological shift ideals}, Bull. Malays. Math. Sci. Soc., {\bf 46}111  (2023) 1--17.
		
		\bibitem{B79a} M. Brodmann, {\it The asymptotic nature of the analytic spread}, Math. Proc. Cambridge Philos. Soc., {\bf 86} (1979), 35--39.
		
	\bibitem{CHL}	L. Chu, J. Herzog and D. Lu, {\it The socle module of a monomial ideal}, Rocky Mountain J. Math., {\bf 51} (2021),  805--821.

\bibitem{CMS} Y. Cid-Ruiz, J. P. Matherne, A. Shapiro, \textit{Syzygies of polymatroidal ideals}, 2025, preprint arxiv.org/abs/2507.13153

		





\bibitem{F1} A. Ficarra, {\it Homological shifts of polymatroidal ideals}, 2024,  preprint { https://arxiv.org/abs/2205.04163}.

\bibitem{F2}  A. Ficarra, Homological shift ideals: Macaulay2 Package, 2023, preprint { https://arxiv. org/abs/2309.09271}.

\bibitem{F3}   A. Ficarra, {\it Shellability of componentwise discrete polymatroids},  Electron. J. Comb., {\bf 32},  (2025), P1.41.

\bibitem{FL} A. Ficarra, D, Lu, {\it Polymatroidal ideals and their asymptotic syzygies}, 2025 preprint {https://arxiv.org/abs/2509.11977}.

\bibitem{FH2023} A. Ficarra, J. Herzog, {\it Dirac's Theorem and Multigraded Syzygies}. Mediterr. J. Math. 20, 134 (2023). https://doi.org/10.1007/s00009-023-02348-8
		
\bibitem{FM1}  A. Ficarra, S. Moradi, {\it Stanley-Reisner ideals with linear powers}, 2025, preprint https://arxiv.org/abs/2508.10354.


\bibitem{FM2}  A. Ficarra, S. Moradi, {\it Complementary edge ideals}, 2025, preprint https://arxiv.org/abs/2508.10870.

\bibitem{FM3}  A. Ficarra, S. Moradi, {\it Rees algebras of complementary edge ideals}, (2025)  preprint https://arxiv.org/abs/2509.18048.

\bibitem{FQ1}	A. Ficarra, A. A. Qureshi, {\it Edge ideals and their asymptotic syzygies}, J. Pure Appl. Algebra,  {\bf 229} (2025), No. 108079, 19 pp.

\bibitem{FQ2} A. Ficarra, A. A. Qureshi, {\it The homological shift algebra of a monomial ideal}, (2024), preprint {https://arxiv.org/abs/2412.21031}.


\bibitem{Froberg88} R. Fr\"oberg, {\it On Stanley-Reisner rings}, Topics in algebra, Part 2 (Warsaw, 1988), 57--70, Banach Center Publ., 26, Part 2, PWN, Warsaw, 1990.
		
\bibitem{HS1973}		F. Harary and A. J. Schwenk, {\it The number of caterpillars}, Discrete Math. 6(4), (1973), 359--365.
 \bibitem{HHO}  J. Herzog, T. Hibi, H. Ohsugi, {\it  Binomial ideals}, Graduate Text in Mathematics {\bf 279}, Springer, (2018).

		\bibitem{HHZ2004} J. Herzog, T. Hibi and X. Zheng, {\it Monomial ideals whose powers have a linear resolution},	Math. Scand., (2004), 23--32.
\bibitem{HHBook} J. Herzog, T. Hibi, {\it Monomial ideals}, Graduate texts in Mathematics {\bf 260}, Springer, 2011.
		\bibitem{HHV} J. Herzog, T. Hibi, M. Vladoiu, \textit{Ideals of fiber type and polymatroids}, Osaka J. Math. 42 (4) (2005) 807--829.
		\bibitem{HMRZ021a} J. Herzog, S. Moradi, M. Rahimbeigi, G. Zhu, {\it Homological shift ideals}. Collect. Math., {\bf 72} (2021), 157--74.
		\bibitem{HRV} J. Herzog, A. Rauf, M. Vladoiu, \textit{The stable set of associated prime ideals of a polymatroidal ideal}, J. Algebr. Comb. {\bf 37} (2013), 289-312.


			\bibitem{HMRZ023} J. Herzog, S. Moradi, M. Rahimbeigi, G. Zhu,  {\it Some homological properties of borel type ideals}, Comm. Algebra {\bf 51} (2023), 1517--1531.

    \bibitem{ET} J. Herzog, Y. Takayama, {\it Resolutions by mapping cones}, in: The Roos Festschrift volume Nr.2(2), Homology, Homotopy and Applications {\bf 4}, (2002), 277--294.
		
	
  \bibitem{HQM}  T. Hibi, A. A. Qureshi, S. Saeedi Madani, {\it Complementary edge ideals}, arXiv:2508.09837.

    	\bibitem{JZ}   A. S. Jahan and X. Zheng, {\it Ideals with linear quotients}, J. Combin. Theory Ser. A,  {\bf 117} (2010),104--110.

	\bibitem{LW} D. Lu, Z. Wang, \textit{The resolutions of generalized co-letterplace ideals and their powers}, Journal of Algebra, 2025, 673, 321--350.
		

		\bibitem{MS}  E. Miller and  B. Sturmfels, {\it Combinatorial Commutative Algebra}, Vol. 227, Springer Science \& Business Media, 2005.

	\bibitem{RS}  A. Roy and  K. Saha,  {\it  On the homological shifts of cover ideals of Cohen-Macaulay graphs},  2025, preprint https://arxiv.org/abs/2506.01810.


	\bibitem{TBR} N. Taghipour, S. Bayati, F. Rahmati, {\it  Homological linear quotients and edge ideals of graphs}, Bull. Aust. Math. Soc., {\bf 110} (2024),  291--302.





\end{thebibliography}
\end{document}